\documentclass[12pt]{article}
\usepackage{a4wide}
\usepackage{amsthm}
\usepackage{amsfonts}
\usepackage{amssymb}
\usepackage{amsmath}
\usepackage{cite}
\usepackage{epsfig}
\usepackage{enumerate}
\usepackage{graphicx}
\newcommand\scalemath[2]{\scalebox{#1}{\mbox{\ensuremath{\displaystyle #2}}}}

\newcommand\NN{{\mathbb N}}
\newcommand\RR{{\mathbb R}}
\newcommand{\dd}{\;\mathrm{d}}
\newcommand{\Ker}{\operatorname{Ker}}

\newtheorem{theorem}{Theorem}
\newtheorem{corollary}[theorem]{Corollary}
\newtheorem{proposition}[theorem]{Proposition}
\newtheorem{lemma}[theorem]{Lemma}

\newtheorem{conjecture}[theorem]{Conjecture}

\DeclareTextCompositeCommand{\v}{OT1}{l}{l\nobreak\hspace{-.1em}'}
\DeclareTextCompositeCommand{\v}{OT1}{t}{t\nobreak\hspace{-.1em}'\nobreak\hspace{-.15em}}

\begin{document}
\title{Forcing quasirandomness with $4$-point permutations}
\author{Daniel Kr{\'a}\v{l}\thanks{Faculty of Informatics, Masaryk University, Botanick\'a 68A, 602 00 Brno, Czech Republic. E-mail: {\tt dkral@fi.muni.cz}. This author was supported by the MUNI Award in Science and Humanities (MUNI/I/1677/2018) of the Grant Agency of Masaryk University.}\and
	Jae-baek Lee\thanks{Department of Mathematics and Statistics, University of Victoria, Victoria, B.C., Canada. E-mail: {\tt dlwoqor0923@uvic.ca}, {\tt noelj@uvic.ca}.}\and
        Jonathan A. Noel\footnotemark[2]${\text{ }}^{\text{,}}$\thanks{Research supported by NSERC Discovery Grant RGPIN-2021-02460, NSERC Early Career Supplement DGECR-2021-00024 and a Start-Up Grant from the University of Victoria.}
	}

\date{} 
\maketitle

\begin{abstract}
A combinatorial object is said to be quasirandom if it exhibits certain 
properties that are typically seen in a truly random object of the same kind. 
It is known that a permutation is quasirandom
if and only if the pattern density of each of the twenty-four $4$-point permutations is close to $1/24$, 
which is its expected value in a random permutation.
In other words, the set of all twenty-four $4$-point permutations is quasirandom-forcing.
Moreover, it is known that there exist sets of eight $4$-point permutations that are also quasirandom-forcing.
Breaking the barrier of linear dependency of perturbation gradients,
we show that every quasirandom-forcing set of $4$-point permutations must have cardinality at least five.
\end{abstract}

\section{Introduction}
\label{sec:intro}

The notion of \emph{quasirandomness} in combinatorics dates back to the 1980s, originating 
from the study of quasirandom graphs by 
R\"odl~\cite{Rod86}, Thomason~\cite{Tho87} and Chung, Graham and Wilson~\cite{ChuGW89}.
As it turned out, several properties that a random graph possesses with 
high probability are satisfied by a large graph if and only if one of them is.
For example, an $n$-vertex graph has $(p+o(1))\binom{n}{2}$ edges and $(p^4+o(1))n^4$
closed walks of length four if and only if
 the induced density of every small subgraph is close to its expected 
value in the Erd\H os--R\'enyi random graph with density $p$.

Results of a similar kind have since been obtained for other types of objects 
such as  
groups~\cite{Gow08},
hypergraphs~\cite{ChuG90,Gow06,Gow07,HavT89,KohRS02},
set systems~\cite{ChuG91s},
sets of integers~\cite{ChuG92}, oriented graphs~\cite{Gri13} and
tournaments~\cite{BucLSS19,ChuG91,CorR17,HanKKMPSV23}.
This leads us to the following metadefinition.
A set $Z$ of combinatorial objects is \emph{quasirandom-forcing} if the following holds:
a sequence $(X_n)_{n\in\NN}$ of objects of increasing size is quasirandom if and only if
the density (for an appropriate notion of density) of every element of $Z$ in $(X_n)_{n\in\NN}$
converges to the expected density of $Z$ in the random object of the same kind.

In this paper, we are interested in quasirandom-forcing 
sets of permutations as studied in~\cite{ChaKNPSV20,Coo04,CruDN,KraP13,Kur22}.
Graham (see~\cite[page 141]{Coo04}) asked
whether there exists $k$ such that the set of all $k$-point permutations is quasirandom-forcing.
We remark that the notion of density considered in this setting
is the pattern density (see Section~\ref{sec:notation} for the definition).
Graham's question was answered affirmatively in~\cite{KraP13}
where it was shown using the methods of the theory of combinatorial limits that any
$k\geq4$ has this property.
Interestingly, this statement is equivalent to a result in 
statistics on non-parametric independence tests by Yanagimoto~\cite{Yan70},
which improved an older result by Hoeffding~\cite{Hoe48}.
The set of all $3$-point permutations is not quasirandom-forcing (see~\cite[Section~4]{KraP13} or~\cite[Proposition~9]{Yan70}) 
and so $k=4$ is best possible in regard to Graham's question.

The set of all twenty-four $4$-point permutations is not the smallest quasirandom-forcing set of permutations.
By inspecting the proof given in~\cite{KraP13},
Zhang~\cite{Zha} observed that there exists a $16$-element quasirandom-forcing set of $4$-point permutations.
Subsequently,
Chan et al.~\cite{ChaKNPSV20} found four distinct $8$-element quasirandom-forcing sets of $4$-point permutations. In fact, 
for each such set, it is enough to test the sum of the pattern densities of the permutations in the set rather
than each individual pattern density.
One of these examples was already known as the Bergsma--Dassios sign covariance in the statistics context~\cite{BerD14,NanWD16}.
Recently, Crudele, Dukes and the third author~\cite{CruDN} found a $6$-element quasirandom-forcing set of permutations
consisting of two $3$-point permutations and four $4$-point permutations.

Regarding lower bounds, Kure\v cka~\cite{Kur22} showed that
every quasirandom-forcing set of permutations has at least four elements. 
Specifically, he used methods from the theory of combinatorial limits to show that
the uniform permuton, the limit object representing any sequence of quasirandom permutations,
can be perturbed to preserve the pattern density of any three fixed permutations.
Kure\v cka's argument is based on establishing the linear independence of the gradients of the pattern densities of perturbed permutons,
which permits using the Implicit Function Theorem to find a suitable perturbation. Our main theorem
strengthens Kure\v cka's result~\cite{Kur22} in the case of $4$-point permutations. 

\begin{theorem}
\label{thm:main}
Every quasirandom-forcing set of $4$-point permutations has cardinality at least five. 
\end{theorem}

We note that the argument of Kure\v cka from~\cite{Kur22} cannot be extended to quadruples of permutations even when all considered permutations have size four
since the associated gradients are often linearly dependent. To address this, we develop a \emph{general method to overcome the linear dependence barrier}:
in Section~\ref{sec:implicit}, we prove a variant of the Implicit Function Theorem
where the assumption on the linear independence of the gradients is relaxed by studying the associated Hessian matrices. 
In Section~\ref{sec:indep}, we classify the sets of four $4$-point permutations for which the gradient-based approach 
of~\cite{Kur22} is insufficient (Theorem~\ref{thm:independent}). Then, in Section~\ref{sec:general} we use the variant of the Implicit Function Theorem 
obtained in Section~\ref{sec:implicit} to obtain general conditions under which a 
set of permutations can be shown to not be quasirandom-forcing.
We then apply this result in Section~\ref{sec:dep} to prove that all 
sets of four $4$-point permutations fail to be quasirandom-forcing (Theorems~\ref{thm:one} and~\ref{thm:zero}),
with the exception of two particular quadruples which need to be handled in a different way.
The two quadruples are then analyzed in Sections~\ref{sec:special1} and~\ref{sec:special2}
by specific arguments (Theorems~\ref{thm:special1} and~\ref{thm:special2}). 
Theorems~\ref{thm:independent}, \ref{thm:one}, \ref{thm:zero}, \ref{thm:special1} and \ref{thm:special2} 
yield Theorem~\ref{thm:main}.

\section{Notation}
\label{sec:notation}

In this section, we fix notation used throughout the paper and
introduce the basic notions related to permutation limits.
The set of the first $n$ positive integers is denoted by $[n]$, i.e., $[n]=\{1,\ldots,n\}$.
We write $f:[m]\nearrow[n]$ to express that $f$ is a non-decreasing function from $[m]$ to $[n]$.
We often work with vectors and we use $\vec 0$ to denote the zero vector;
the dimension of the vector will always be clear from context.
We say that coefficients $\alpha_1,\ldots,\alpha_k$ are \emph{non-trivial} if at least one of them is non-zero;
in particular,
vectors $v_1,\ldots,v_k$ are linearly dependent if and only if
there exist non-trivial coefficients $\alpha_1,\ldots,\alpha_k$ such that
$\alpha_1 v_1+\cdots+\alpha_k v_k=\vec 0$.
A $(k\times\ell)$-matrix is a matrix with $k$ rows and $\ell$ columns.
Finally, a matrix $A$ is an \emph{all one matrix} if every entry of $A$ is equal to one, and
a square matrix $A$ is \emph{doubly stochastic}
if its entries are non-negative and
the sum of the entries in each row and in each column of $A$ is equal to one.

A permutation $\pi$ of \emph{size} $k$ is a bijection from $[k]$ to $[k]$;
we often refer to a permutation of size $k$ as a \emph{$k$-point} permutation.
The size of a permutation $\pi$ is denoted by $|\pi|$.
When $\pi$ is a $k$-point permutation, we usually write $\pi(1)\cdots\pi(k)$ to denote $\pi$.
For example, $132$ is the $3$-point permutation $\pi$ such that $\pi(1)=1$, $\pi(3)=3$ and $\pi(2)=2$.
A $k$-point permutation $\pi$ is often visualized in a $k\times k$ grid
with the cells at positions $(i,\pi(i))$, $i\in [k]$, marked;
the cell at the position $(1,1)$ is the cell in the bottom left corner.
An example can be found in Figure~\ref{fig:visual}.

\begin{figure}
\begin{center}
\epsfbox{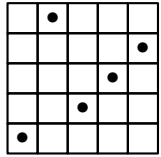}
\end{center}
\caption{Visualization of the permutation $15234$.}
\label{fig:visual}
\end{figure}

The \emph{subpermutation} of a $k$-point permutation $\pi$ \emph{induced} by a set $A\subseteq [k]$
is the unique $|A|$-point permutation $\sigma$ such that
$\sigma(i)<\sigma(j)$ if and only if $\pi(a_i)<\pi(a_j)$ for all $i,j\in [|A|]$
where $A=\{a_1,\ldots,a_{|A|}\}$ and $a_1<\cdots<a_{|A|}$.
For example,
the subpermutation of $15234$ induced by the set $\{1,2,4\}$ is $132$.
The \emph{pattern density} of a permutation $\sigma$ in a permutation $\pi$ with $|\pi|\ge|\sigma|$
is the probability that subpermutation of $\pi$ induced by a uniformly random $|\sigma|$-element subset of $[|\pi|]$ is $\sigma$.
In the rest of the paper,
we will just use the word \emph{density} as opposed to pattern density in the interest of brevity.
The density of a permutation $\sigma$ in a permutation $\pi$ is denoted by $d(\sigma,\pi)$;
for completeness, we set $d(\sigma,\pi)=0$ if $|\sigma|>|\pi|$.

We next introduce basic concepts from the theory of permutation limits developed in~\cite{HopKMRS13,HopKMS11a},
where the limit objects are represented by measures with uniform marginals as in~\cite{KraP13}.
A sequence of permutations $(\pi_n)_{n\in\NN}$ is \emph{convergent}
if the sizes of $\pi_n$ tend to infinity and
the sequence $d(\sigma,\pi_n)$ converges for every permutation $\sigma$.
A \emph{permuton} is a Borel probability measure on $[0,1]^2$ such that
its projection to either of the two axes is the uniform measure (this property is referred to as having \emph{uniform marginals}).
If $\mu$ is a permuton,
\emph{$\mu$-random permutation of size $k$} is obtained as follows.
First sample $k$ points in $[0,1]^2$ according to the probability measure $\mu$.
Note that
the $x$-coordinates and the $y$-coordinates of the $k$ sampled points are distinct with probability one
since $\mu$ has uniform marginals.
If $(x_1,y_1),\ldots,(x_k,y_k)$ are the sampled points listed in a way that $x_1<\cdots<x_k$,
then the $\mu$-random permutation $\sigma$
is the unique permutation such that $\sigma(i)<\sigma(j)$ if and only if $y_i<y_j$ for all $i,j\in [k]$.
Finally,
the density of a permutation $\sigma$ in a permuton $\mu$, denoted by $d(\sigma,\mu)$,
is the probability that the $\mu$-random permutation of size $|\sigma|$ is $\sigma$.
It can be shown~\cite{HopKMRS13,HopKMS11a} that
if $(\pi_n)_{n\in\NN}$ is a convergent sequence of permutations,
then there exists a unique permuton $\mu$ such that
$d(\sigma,\mu)$ is the limit of $d(\sigma,\pi_n)$ for every permutation $\sigma$;
this permuton $\mu$ is called the \emph{limit} of $(\pi_n)_{n\in\NN}$.

The \emph{uniform permuton} is the uniform measure $\lambda$ on $[0,1]^2$.
Recall that a sequence $(\pi_n)_{n\in\NN}$ of permutations is \emph{quasirandom}
if the limit of $d(\sigma,\pi_n)$ is equal to $1/|\sigma|!$ for every permutation $\sigma$,
i.e., the limit of $(\pi_n)_{n\in\NN}$ is the uniform permuton.

We frequently examine permutons created by partitioning $[0,1]^2$ into a square grid 
and assigning uniform measures with varying densities to these squares. More formally, 
given a doubly stochastic $(k\times k)$-matrix $A$,
we define $\mu[A]$ to be the permuton such that
\[\mu[A](X)=k\cdot\sum_{i,j\in [k]} A_{ij} \lambda\left(X\cap\left[\frac{i-1}{k},\frac{i}{k}\right)\times\left[\frac{j-1}{k},\frac{j}{k}\right)\right)
  \]
for any Borel set $X\subseteq [0,1]^2$
where $\lambda$ is the uniform Borel measure on $[0,1]^2$.
We say that a permuton $\mu$ is a \emph{step permuton} with $k$ steps
if $\mu$ is equal to $\mu[A]$ for some doubly stochastic $(k\times k)$-matrix $A$.

We will make use of the following formula for $d(\sigma,\mu[A])$ where $\sigma$ is a permutation and 
$A$ is a doubly stochastic $(k\times k)$-matrix, which was proven in~\cite[Lemma 10]{ChaKNPSV20}:
\[d(\sigma,\mu[A])=\frac{|\sigma|!}{k^{|\sigma|}}\sum_{f,g:[|\sigma|]\nearrow [k]}
  \frac{\prod_{i\in [|\sigma|]}A_{f(i),g(\sigma(i))}}{\prod\limits_{i\in [k]}\lvert f^{-1}(i)\rvert!\cdot\lvert g^{-1}(i)\rvert!}.
\]  
Fix $k\ge 2$. For $j,j'\in [k-1]$, let $Z^{j,j'}$ be the $(k\times k)$-matrix defined as
\[Z^{j,j'}_{i,i'}=\begin{cases}
                  +1 & \mbox{if $(i,i')=(j,j')$ or $(i,i')=(j+1,j'+1)$,}\\
                  -1 & \mbox{if $(i,i')=(j+1,j')$ or $(i,i')=(j,j'+1)$, and}\\
                  0 & \mbox{otherwise.}
		  \end{cases}\]
Let $Z^0$ be the all one $(k\times k)$-matrix and
define $B(x_1,\ldots,x_{(k-1)^2})$ to be the following $(k\times k)$-matrix:
\[B\left(x_1,\ldots,x_{(k-1)^2}\right)=Z^0+x_1Z^{1,1}+x_2Z^{1,2}+\cdots+x_{k-1}Z^{1,k-1}+x_kZ^{2,1}+\cdots+x_{(k-1)^2}Z^{k-1,k-1}.\]
For example, if $k=3$, then $B(x_1,x_2,x_3,x_4)$ is
\[\begin{bmatrix}
                     1 & 1& 1\\
                     1 & 1& 1\\
                     1 & 1& 1
		     \end{bmatrix} 
       + \begin{bmatrix}
                     x_1 & -x_1& 0\\
                     -x_1 & x_1& 0\\
                     0 & 0& 0
		     \end{bmatrix}
       + \begin{bmatrix}
                     0 & x_2& -x_2\\
                     0 & -x_2& x_2\\
                     0 & 0& 0
		     \end{bmatrix}
       + \begin{bmatrix}
                     0 & 0& 0\\
                     x_3 & -x_3& 0\\
                     -x_3 & x_3& 0
		     \end{bmatrix}
       + \begin{bmatrix}
                     0 & 0& 0\\
                     0 & x_4& -x_4\\
                     0 & -x_4& x_4
		     \end{bmatrix}
       \]
\[=\begin{bmatrix}
                     1+x_1 & 1-x_1+x_2 & 1-x_2 \\
		     1-x_1+x_3 & 1+x_1-x_2-x_3+x_4 & 1+x_2-x_4 \\
		     1-x_3 & 1+x_3-x_4 & 1+x_4
		     \end{bmatrix}.
\]
Note that each row and column of $B\left(x_1,\ldots,x_{(k-1)^2}\right)$ sums to $k$ by construction. 
Thus, if the values of $x_1,\ldots,x_{(k-1)^2}$ are at most $1/4$ in absolute value,
then $\frac{1}{k}B\left(x_1,\ldots,x_{(k-1)^2}\right)$ is doubly stochastic.
For a permutation $\sigma$,
we define  $h^k_{\sigma}:\RR^{(k-1)^2}\to\RR$ by
\[h^k_{\sigma}\left(x_1,\ldots,x_{(k-1)^2}\right)=-\frac{k^{2|\sigma|}}{|\sigma|!}+|\sigma|!\sum_{f,g:[|\sigma|]\nearrow [k]}
  \frac{\prod_{i\in [|\sigma|]}B\left(x_1,\ldots,x_{(k-1)^2}\right)_{f(i),g(\sigma(i))}}{\prod\limits_{i\in [k]}\lvert f^{-1}(i)\rvert!\cdot\lvert g^{-1}(i)\rvert!}.
  \]  
Note that
\[d\left(\sigma,\mu\left[\frac{1}{k}B\left(x_1,\ldots,x_{(k-1)^2}\right)\right]\right)=\frac{1}{|\sigma|!}+\frac{1}{k^{2|\sigma|}}h^k_{\sigma}\left(x_1,\ldots,x_{(k-1)^2}\right).
  \]
In particular, it holds that $h^k_{\sigma}(\vec 0)=0$ for every permutation $\sigma$ and every integer $k\ge 2$. 
Our main strategy for showing that a quadruple $\pi_1,\pi_2,\pi_3,\pi_4$ of $4$-point permutations fails 
to force quasirandomness is to find a non-zero vector $x\in[-1/4,1/4]^{(k-1)^2}$ 
such that $h^k_{\pi_i}(x)=0$ for each $i\in[4]$. The next lemma implies that this is sufficient.

\begin{lemma}
\label{lem:h=0}
Let $\pi_1,\dots,\pi_m$ be permutations. If there exists $k\geq2$ and a non-zero vector $x\in[-1/4,1/4]^{(k-1)^2}$ 
such that $h^k_{\pi_i}(x)=0$ for all $i\in[m]$, then $\{\pi_1,\dots,\pi_m\}$ is not quasirandom-forcing.
\end{lemma}

\begin{proof}
Note that all entries of $B(x)=B\left(x_1,\dots,x_{(k-1)^2}\right)$ are non-negative since $x\in[-1/4,1/4]^{(k-1)^2}$. For $i\in[m]$, the fact that $h^k_{\pi_i}(x)=0$ tells us that $d\left(\pi_i,\mu\left[\frac{1}{k}B(x)\right]\right)=\frac{1}{|\pi_i|!}$. So, all that remains is to show that $\mu\left[\frac{1}{k}B(x)\right]$ is not the uniform measure, which is the same as showing that $B(x)$ is not an all one matrix. 
Since $x$ is a non-zero vector, we can take $t$ to be the minimum index such that $x_t\neq 0$.
Let $i,j\in[k-1]$ so that $t=(i-1)(k-1)+j$.
Then, by the choice of $t$,
the entry of $B(x)$ on the $i$-th row and $j$-th column is equal to $1+x_t\neq 1$.
\end{proof}

In our arguments in Section~\ref{sec:dep},
we will study the gradient of the function $h^k_{\sigma}$, i.e., $\nabla h^k_{\sigma}\left(x_1,\ldots,x_{(k-1)^2}\right)$, and
the Hessian matrix of the function $h^k_{\sigma}$,
which will be denoted by $H^k_{\sigma}\left(x_1,\ldots,x_{(k-1)^2}\right)$.
We remark that the reason for defining the function $h^k_{\sigma}$ as above rather than simply
as the difference $d\left(\sigma,\mu\left[\frac{1}{k}B\left(x_1,\ldots,x_{(k-1)^2}\right)\right]\right)-\frac{1}{|\sigma|!}$
is that the coefficients of $\nabla h^k_{\sigma}(\vec 0)$ and $H^k_{\sigma}(\vec 0)$
for $4$-point permutations $\sigma$ turn out to be half-integral, which makes them simpler to present.

\section{Implicit Function Theorem}
\label{sec:implicit}

In this section,
we prove a variant of the Implicit Function Theorem that we will use in this paper.
We start with recalling the statement of the Implicit Function Theorem.

\begin{theorem}[Implicit Function Theorem]
\label{thm:implicit}
Let $F(x_1,\ldots,x_n,z):\RR^{n+1}\to\RR^n$ be a continuously differentiable function such that $F(\vec 0)=\vec 0$.
If the Jacobian matrix
\[J=\left(\frac{\partial}{\partial x_j}F_i(\vec 0)\right)_{i,j\in [n]}\]
is invertible,
then there exist $\varepsilon_0>0$ and
a unique continuously differentiable function $g:[-\varepsilon_0,\varepsilon_0]\to\RR^n$ such that
$F(g(z),z)=\vec 0$ for all $z\in [-\varepsilon_0,\varepsilon_0]$.
\end{theorem}

The following corollary of the Implicit Function Theorem,
which serves as a motivation for proving the variant of the Implicit Function Theorem stated later.

\begin{corollary}
\label{cor:implicit}
Let $F(x_1,\ldots,x_m,z):\RR^{m+1}\to\RR^n$, $m\geq n$, be a continuously differentiable function such that $F(\vec 0)=\vec 0$.
If the Jacobian matrix
\[J=\left(\frac{\partial}{\partial x_j}F_i(\vec 0)\right)_{i\in [n],j\in [m]}\]
has rank $n$,
then there exist $\varepsilon_0>0$ and
a continuously differentiable function $g:[-\varepsilon_0,\varepsilon_0]\to\RR^m$ such that
$F(g(z),z)=\vec 0$ for all $z\in [-\varepsilon_0,\varepsilon_0]$.
\end{corollary}

In our arguments, we will use the following extension of Corollary~\ref{cor:implicit};
note that Theorem~\ref{thm:extended} extends the Implicit Function Theorem to the setting
when the rank of the Jacobian is one less than its dimension.
Unfortunately,
we have not been able to locate a suitable reference in the literature for it.
So, we provide a short proof for completeness.
We remark that it is possible to show that the function $g$ can be chosen to be continuous
but since we do not need this in our arguments,
we decided to omit the proof of this property to keep the proof as simple as possible.

\begin{theorem}
\label{thm:extended}
Let $F(x_1,\ldots,x_m,z):\RR^{m+1}\to\RR^n$, $m\geq n$, be a smooth function such that $F(\vec 0)=\vec 0$.
Suppose that the rank of the Jacobian matrix
\[J=\left(\frac{\partial}{\partial x_j}F_i(\vec 0)\right)_{i\in [n],j\in [m]}\]
is $n-1$.
Let $w\in\RR^n$ be any non-zero vector in $\Ker J^T$ and
let $F^w:\RR^m\to\RR$ be the function defined by
\[F^w(x_1,\ldots,x_m)=\sum_{i\in [n]}w_i\cdot F_i(x_1,\ldots,x_m,0).\]
Let $R:\RR^{m-(n-1)}\to\Ker J$ be an isomorphism, and
let $G:\RR^{m-(n-1)}\to\RR$ be the function
defined by 
\[G(y_1,\dots,y_{m-(n-1)}) = F^w\left(R(y_1,\dots,y_{m-(n-1)})\right).\]
If the Hessian matrix
\[H=\left(\frac{\partial}{\partial y_{i} \partial y_{j}}G(\vec 0)\right)_{i,j\in [m-(n-1)]}\]
has both a positive eigenvalue and a negative eigenvalue,
then there exist $\varepsilon_0>0$ and
a function $g:[-\varepsilon_0,\varepsilon_0]\to\RR^m$ such that
$F(g(z),z)=\vec 0$ for all $z\in [-\varepsilon_0,\varepsilon_0]$.
\end{theorem}

\begin{proof}
By using the fact that the Jacobian matrix $J$ has rank $n-1$ and applying a suitable linear transformation on $\RR^m$
we may assume the following:
\begin{enumerate}[(a)]
\item\label{linearTerms} $\left(\frac{\partial}{\partial x_i}F_1(\vec 0),\dots,\frac{\partial}{\partial x_i}F_n(\vec 0)\right)^T$ is the $i$-th standard unit vector for $i\in [n-1]$,
\item\label{linearTerms2} $\left(\frac{\partial}{\partial x_i}F_1(\vec 0),\dots,\frac{\partial}{\partial x_i}F_n(\vec 0)\right)^T=\vec{0}$ for all $i\in\{n,\dots,m+1\}$,
\item\label{wchoice} $w$ is the $n$-th standard unit vector.
\end{enumerate}
By \eqref{wchoice}, we have $F^w=F_n$. 
Note that \eqref{linearTerms} determines the first $n-1$ columns of the Jacobian matrix $J$ and
\eqref{linearTerms2} the remaining $m-(n-1)$ columns.
Hence,
we may assume that $R(y_1,\ldots,y_{m-(n-1)})=(0,\ldots,0,y_1,\ldots,y_{m-(n-1)})$ and
so $G(y_1,\ldots,y_{m-(n-1)})=F_n(0,\ldots,0,y_1,\ldots,y_{m-(n-1)},0)$.
It follows that the matrix $H$ can be equivalently expressed as
\[H=\left(\frac{\partial}{\partial x_{i+(n-1)} \partial x_{j+(n-1)}}F_n(\vec 0)\right)_{i,j\in [m-(n-1)]}.\]
Now, using the fact that $H$ has a positive and a negative eigenvalue, we can additionally assume that
\begin{enumerate}[(a)]
\addtocounter{enumi}{3}
\item\label{quadTerms} $\frac{\partial}{\partial x_n\partial x_n}F_n(\vec 0)$ is equal to $1$,
\item\label{quadTerms2} $\frac{\partial}{\partial x_{n+1}\partial x_{n+1}}F_n(\vec 0)$ is equal to $-1$, and
\item\label{quadTerms3} $\frac{\partial}{\partial x_n\partial x_{n+1}}F_n(\vec 0)$ is zero.
\end{enumerate}
 
Let us pause the presentation of the proof to digest what the above statements say about the function $F$. 
By \eqref{linearTerms} and \eqref{linearTerms2}, the Taylor series of $F_i$ with respect to $x_1,\ldots,x_{n+1}$
contains a single linear term, namely $x_i$, for each $i\in[n-1]$ and
the Taylor series of $F_n$ contains no linear terms.
It follows that there exist $K\geq1$ and $\varepsilon\in (0,1/2)$ 
such that, for all $i\in [n-1]$ and $(x_1,\ldots,x_{n+1})\in (-\varepsilon,\varepsilon)^{n+1}$,
\begin{equation}
\left|F_i(x_1,\ldots,x_{n+1},0,\ldots,0)-x_i\right| \le  K\cdot\max\left\{x_1^2,\ldots,x_{n+1}^2\right\}.
\label{eq:F_i=x_i}
\end{equation}
Moreover, \eqref{quadTerms}--\eqref{quadTerms3} imply that the 
Taylor series of $F_n$ contains the terms $x_n^2$ and $-x_{n+1}^2$ and does not contain 
the cross term $x_nx_{n+1}$. Therefore, we can additionally assume that $K$ and $\varepsilon$ have been chosen in
such a way that, for all $(x_1,\ldots,x_{n+1})\in (-\varepsilon,\varepsilon)^{n+1}$,
\begin{equation}
\scalemath{0.95}{
\left|F_n(x_1,\ldots,x_{n+1},0,\ldots,0)-x_n^2+x_{n+1}^2\right| \le  K\cdot\max\left\{|x_ix_j|: i,j\in[n+1], \{i,j\}\not=\{n,n+1\}\right\}.}
\label{eq:F_n}
\end{equation}
For each $z\in (-\varepsilon,\varepsilon)$, let $F^z:\RR^{n}\to\RR^n$ be the function defined by
\[F^z(x_1,\ldots,x_n)=F\left(\frac{\varepsilon x_1}{2K^3},\ldots,\frac{\varepsilon x_{n-1}}{2K^3},
                             \max\left\{\frac{\varepsilon x_n}{2K^2},0\right\},
			     \min\left\{\frac{\varepsilon x_n}{2K^2},0\right\},
			     0,\ldots,0,z\right).\]
We show that the point $\vec 0\in\RR^n$ is in the interior of $F^0([-1,1]^n)$.
Consider a point $(x_1,\ldots,x_n)$ on the boundary of $[-1,1]^n$.
If $x_i=1$ for some $i\in[n-1]$, then, by \eqref{eq:F_i=x_i},
\[F^0_i(x_1,\ldots,x_n)\ge \frac{\varepsilon}{2K^3}-K\left(\frac{\varepsilon}{2K^2}\right)^2\ge \frac{\varepsilon}{4K^3}.\]
Similarly, if $x_i=-1$ for some $i\in[n-1]$, then
\[F^0_i(x_1,\ldots,x_n)\le -\frac{\varepsilon}{2K^3}+K\left(\frac{\varepsilon}{2K^2}\right)^2\le -\frac{\varepsilon}{4K^3}.\]
Next, if $x_n=1$, then, by \eqref{eq:F_n},
\[F^0_n(x_1,\ldots,x_n)\ge \left(\frac{\varepsilon}{2K^2}\right)^2-K\left(\frac{\varepsilon}{2K^2}\right)\left(\frac{\varepsilon}{2K^3}\right)\ge \frac{\varepsilon}{4K^4}.\]
Similarly, if $x_n=-1$, then
\[F^0_n(x_1,\ldots,x_n)\le -\left(\frac{\varepsilon}{2K^2}\right)^2+K\left(\frac{\varepsilon}{2K^2}\right)\left(\frac{\varepsilon}{2K^3}\right)\le -\frac{\varepsilon}{4K^4}.\]
We conclude that
the boundary the $F^0$-image of the boundary of $[-1,1]^n$ encircles the point $\vec 0\in\RR^n$;
note that
we have actually proven that the distance of the point $\vec 0$ from the $F^0$-image of the boundary of $[-1,1]^n$, and therefore
from the boundary of $F^0([-1,1]^n)$, is at least $\frac{\varepsilon}{4K^4}$.
Using that the function $F^0$ is continuous,
we deduce that the point $\vec 0$ is in the interior of $F^0([-1,1]^n)$.

Since $\vec 0$ is in the interior of $F^0([-1,1]^n)$ and $F$ is a smooth function,
there exists $\varepsilon_0\in (0,\varepsilon)$ such that
$\vec 0$ is in the interior of $F^z([-1,1]^n)$ for every $z\in [-\varepsilon_0,\varepsilon_0]$.
It follows that for every $z\in [-\varepsilon_0,\varepsilon_0]$,
there exist $x_1,\ldots,x_n\in [-1,1]$ such that $F^z(x_1,\ldots,x_n)=0$ and
so we can set
\[g(z)=\left(\frac{\varepsilon x_1}{2K^3},\ldots,\frac{\varepsilon x_{n-1}}{2K^3},
                   \max\left\{\frac{\varepsilon x_n}{2K^2},0\right\},
                   \min\left\{\frac{\varepsilon x_n}{2K^2},0\right\},
                   0,\ldots,0\right)\]
for any $x_1,\ldots,x_n\in [-1,1]$ such that $F^z(x_1,\ldots,x_n)=0$.		   
The existence of the function $g$ from the statement of the theorem is now established.
\end{proof}

The condition in Theorem~\ref{thm:extended} that the Hessian of the function $G$
has a positive eigenvalue and a negative eigenvalue can be rather technical to verify.
The following corollary of Theorem~\ref{thm:extended} gives a more straightforward condition to check.

\begin{corollary}
\label{cor:extended}
Let $F(x_1,\ldots,x_m,z):\RR^{m+1}\to\RR^n$, $m\ge n$, be a smooth function such that $F(\vec 0)=\vec 0$.
Suppose that the rank of the Jacobian matrix
\[J=\left(\frac{\partial}{\partial x_j}F_i(\vec 0)\right)_{i\in [n],j\in [m]}\]
is $n-1$, and
let $w\in\RR^n$ be a non-zero vector in $\Ker J^T$. Let
$F^w:\RR^{m}\to\RR$ be the function defined as
\[F^w(x_1,\ldots,x_m)=\sum_{i\in [n]}w_i\cdot F_i(x_1,\ldots,x_m).\]
If the Hessian matrix
\[H=\left(\frac{\partial}{\partial x_i \partial x_j}F^w(\vec 0)\right)_{i,j\in [m]}\]
has at least $n$ positive eigenvalues and at least $n$ negative eigenvalues,
then there exist $\varepsilon_0>0$ and
a function $g:[-\varepsilon_0,\varepsilon_0]\to\RR^m$ such that
$F(g(z),z)=\vec 0$ for all $z\in [-\varepsilon_0,\varepsilon_0]$.
\end{corollary}

\begin{proof}
It suffices to verify that the Hessian matrix of the function $G$ defined as in 
Theorem~\ref{thm:extended} has both a positive eigenvalue and a negative eigenvalue.
Let $\lambda_1,\ldots,\lambda_n$ be any $n$ positive eigenvalues of the Hessian matrix of $F^w$ and
$v_1,\ldots,v_n\in\mathbb{R}^m$ the corresponding eigenvectors;
we may assume that $v_1,\ldots,v_n$ are orthogonal unit vectors.
Since $\Ker J$ has dimension $m-(n-1)$, we can let 
$\alpha_1,\ldots,\alpha_n$ be coefficients such that
$v_0:=\alpha_1v_1+\cdots+\alpha_nv_n$ is a non-zero vector in $\Ker J$.
Then $v_0^T H v_0=\lambda_1\alpha_1^2+\cdots+\lambda_n\alpha_n^2$, which is positive.
This implies that the Hessian matrix of the function $G$ has a positive eigenvalue.
The proof that the Hessian matrix of $G$ has a negative eigenvalue is completely analogous. 
\end{proof}

\section{Linearly independent gradients}
\label{sec:indep}

In this section, we analyze quadruples of $4$-point permutations that
are amenable to existing methods.
To analyze such quadruples,
we recall the results from~\cite{Kur22} concerning infinitesimal perturbations of the uniform permuton.
For a permutation $\pi\in S_k$,
we define a \emph{gradient polynomial} $P_\pi:(0,1)^2\to\RR$ as
\begin{align*}
P_\pi(\alpha,\beta)  = & k!\sum_{m\in [k]}\left(\frac{k-m}{1-\alpha}-\frac{m-1}{\alpha}\right)
                                        \left(\frac{k-\pi(m)}{1-\beta}-\frac{\pi(m)-1}{\beta}\right)\\
		     	&		 \frac{\alpha^{m-1}(1-\alpha)^{k-m}\beta^{\pi(m)-1}(1-\beta)^{k-\pi(m)}}{(m-1)!(k-m)!(\pi(m)-1)!(k-\pi(m))!}.
\end{align*}			
Informally speaking,
the value of the gradient polynomial $P_\pi(\alpha,\beta)$
represents the change of a density of a permutation $\pi$ in the uniform permuton
when an infinitesimal elementary perturbation is performed at the point $(\alpha,\beta)\in (0,1)^2$,
i.e., a perturbation that would be represented by a $(K\times K)$-matrix $Z^{\alpha K,\beta K}$ for a very large $K$.
Hence, the gradient polynomial $P_\pi(\alpha,\beta)$ can be viewed as the limit version of $h^n_{\pi}$.
This connection can be made formal as follows~\cite{Kur22}:
if $S$ a set of permutations such that
the gradient polynomials $P_{\pi}$, $\pi\in S$, are linearly independent,
then there exists $n\in\NN$ such that the gradients of $h^n_{\pi}$, $\pi\in S$, are linearly independent.
It is then possible to use Corollary~\ref{cor:implicit} to show that
such a set $S$ is not quasirandom-forcing as stated in the next lemma.

\begin{lemma}[{Kure\v cka~\cite[Lemma 8]{Kur22}}]
\label{lm:gradpoly}
Let $S$ be a set of permutations.
If the gradient polynomials $P_{\pi}$, $\pi\in S$, are linearly independent,
then $S$ is not quasirandom-forcing.
\end{lemma}

The next lemma gives a key property of permutations of the same size with linearly dependent gradient polynomials.
Recall that the permutation matrix $A_{\pi}$ for a $k$-point permutation $\pi$
is a $(k\times k)$-matrix with the entry in the $i$-th row and the $\pi(i)$-th column equal to one and
the remaining entries equal to zero.

\begin{lemma}[{Kure\v cka~\cite[Lemma 12]{Kur22}}]
\label{lm:gradcover}
Let $S=\{\pi_1,\ldots,\pi_m\}$ be a set of $m$ permutations of size $k$.
If the polynomial $t_1P_{\pi_1}+\cdots+t_mP_{\pi_m}$ for some reals $t_1,\ldots,t_m$, is identically equal to zero,
then the combination $t_1A_{\pi_1}+\cdots+t_mA_{\pi_m}$ of the associated permutation matrices
is a constant $(k\times k)$-matrix.
\end{lemma}

We are now ready to prove the main theorem of this section.

\begin{theorem}
\label{thm:independent}
If a set $\{\pi_1,\pi_2,\pi_3,\pi_4\}$ of four $4$-point permutations is quasirandom-forcing,
then either the matrix $A_{\pi_1}+A_{\pi_2}+A_{\pi_3}+A_{\pi_4}$ is the all one matrix or,
possibly after relabeling the four permutations,
the matrix $A_{\pi_1}+A_{\pi_2}-A_{\pi_3}-A_{\pi_4}$ is the zero matrix.
\end{theorem}

\begin{proof}
Suppose that $\{\pi_1,\pi_2,\pi_3,\pi_4\}$ is a set of four distinct 
$4$-point permutations which is quasirandom-forcing.
By Lemma~\ref{lm:gradpoly},
the gradient polynomials $P_{\pi_1},\ldots,P_{\pi_4}$ are linearly dependent,
i.e., there exist non-trivial coefficients $t_1,\ldots,t_4$ such that
$t_1P_{\pi_1}+\cdots+t_4P_{\pi_4}$ is identically zero.
It follows from Lemma~\ref{lm:gradcover} that
the matrix $t_1A_{\pi_1}+\cdots+t_4A_{\pi_4}$ is constant.

If $A_{\pi_1}+A_{\pi_2}+A_{\pi_3}+A_{\pi_4}$ is the all one matrix,
then we are at one of the two conclusions of the theorem. So, from here forward,
we assume that this is not the case.
It follows that $A_{\pi_1}+A_{\pi_2}+A_{\pi_3}+A_{\pi_4}$ has an entry equal to zero. Since
this entry is zero in each of the matrices $A_{\pi_1},A_{\pi_2},A_{\pi_3},A_{\pi_4}$, it is 
also zero in $t_1A_{\pi_1}+\cdots+t_4A_{\pi_4}$, and so $t_1A_{\pi_1}+\cdots+t_4A_{\pi_4}$ must be the zero matrix.

We first show that none of the coefficients $t_1,t_2,t_3$ or $t_4$ is zero.
Suppose that $t_4=0$.
Since the coefficients are not all zero,
at least two of them must be non-zero in order for $t_1A_{\pi_1}+\cdots+t_4A_{\pi_4}$ to be the zero matrix.
So, let us assume, without loss of generality, that $t_1,t_2\neq 0$.
Since $\pi_1\neq \pi_2$, there must exist $i\in[4]$ such that $\pi_1(i)\neq \pi_2(i)$.
Therefore, the $i$-th column of the matrix $t_1A_{\pi_1}+\cdots+t_4A_{\pi_4}$
either contains a non-zero entry in the $\pi_1(i)$-th row or the $\pi_2(i)$-th row,
which is a contradiction. 

Since all four coefficients $t_1,t_2,t_3,t_4$ are non-zero and
the four permutations $\pi_1,\ldots,\pi_4$ are distinct,
there exists $i\in [4]$ such that $\pi_1(i)\neq\pi_2(i)$. 
In order for the $i$-th column of $t_1A_{\pi_1}+\cdots+t_4A_{\pi_4}$ 
to be the zero vector, it must hold (after possibly swapping $\pi_3$ and $\pi_4$) that
$\pi_1(i)=\pi_3(i)$, $t_1=-t_3$, $\pi_2(i)=\pi_4(i)$ and $t_2=-t_4$.
So, two of the coefficients are positive and two are negative.
Without loss of generality, we can assume that both $t_1$ and $t_2$ are positive.
Since the permutations $\pi_1$ and $\pi_3$ are distinct,
there exists $j\in [4]$ such that $\pi_1(j)\not=\pi_3(j)$,
which implies that $\pi_1(j)=\pi_4(j)$ and so $t_1=-t_4$.
It follows that $t_1=t_2=-t_3=-t_4$,
which yields that the matrix $A_{\pi_1}+A_{\pi_2}-A_{\pi_3}-A_{\pi_4}$ is the zero matrix.
\end{proof}

We remark that, as shown in~\cite[Lemma 12]{ChaKNPSV20},
if $\pi_1,\ldots,\pi_4$ are four $4$-point permutations such that
$A_{\pi_1}+\cdots+A_{\pi_4}$ is the all one matrix,
then $\nabla h^n_{\pi_1}(\vec 0),\ldots,\nabla h^n_{\pi_4}(\vec 0)$ are not linearly independent for any $n$;
in fact, the sum of these gradients,
i.e., $\nabla h^n_{\pi_1}(\vec 0)+\cdots+\nabla h^n_{\pi_4}(\vec 0)$,
is equal to $\vec 0$ for every $n\in\NN$.
Similarly,
it can be shown that, if $A_{\pi_1}+A_{\pi_2}-A_{\pi_3}-A_{\pi_4}$ is the zero matrix,
then $\nabla h^n_{\pi_1}(\vec 0)+\nabla h^n_{\pi_2}(\vec 0)-\nabla h^n_{\pi_3}(\vec 0)-\nabla h^n_{\pi_4}(\vec 0)$
is equal to $\vec 0$ for every $n\in\NN$.
Hence, the condition in Theorem~\ref{thm:independent} is the best possible that
can be obtained using methods solely based on Corollary~\ref{cor:implicit},
i.e., the standard version of the Implicit Function Theorem. This motivates the application of the 
variants of the Implicit Function Theorem established from Section~\ref{sec:implicit} in the next section. 

\section{Linearly dependent gradients---general arguments}
\label{sec:general}

In this section,
we present a general theorem, which is based on Theorem~\ref{thm:extended}, that
guarantees that a $k$-tuple of permutations is not quasirandom-forcing.

\begin{theorem}
\label{thm:non-forcing-kernel}
Let $\pi_1,\ldots,\pi_k$ be a $k$-tuple of permutations and $n\in\NN$ such that
$(n-1)^2\ge k+2$ and
$\nabla h^n_{\pi_1}(\vec 0),\ldots,\nabla h^n_{\pi_{k-1}}(\vec 0)$ are linearly independent but
$\alpha_1 \nabla h^n_{\pi_1}(\vec 0)+\cdots+\alpha_k \nabla h^n_{\pi_k}(\vec 0)$ is the zero vector
for non-trivial coefficients $\alpha_1,\ldots,\alpha_k$.
If there exist vectors $w^+$ and $w^-$ that are orthogonal to each $\nabla h^n_{\pi_i}(\vec 0)$, $i\in [k-1]$,
such that
\begin{align*}
0 & < (w^+)^T\left(\alpha_1 H^n_{\pi_1}(\vec 0)+\cdots+\alpha_k H^n_{\pi_k}(\vec 0)\right) w^+ \\
0 & > (w^-)^T\left(\alpha_1 H^n_{\pi_1}(\vec 0)+\cdots+\alpha_k H^n_{\pi_k}(\vec 0)\right) w^-
\end{align*}
then $\{\pi_1,\ldots,\pi_k\}$ is not quasirandom-forcing.
\end{theorem}

\begin{proof}
Fix the $k$-tuple of permutations $\pi_1,\ldots,\pi_k$ and $n\in\NN$
with the properties given in the statement of the theorem.
Let $\alpha_1,\ldots,\alpha_k$ be the non-trivial coefficients such that
$\alpha_1 \nabla h^n_{\pi_1}(\vec 0)+\cdots+\alpha_k \nabla h^n_{\pi_k}(\vec 0)$ is the zero vector and
let $w^+$ and $w^-$ be the vectors with the properties given in the statement;
we can assume that both $w^+$ and $w^-$ are unit vectors that are orthogonal to each other.
Let $A$ be an $(n-1)^2\times (n-1)^2$ orthonormal matrix such that
the first $k-1$ columns form an orthonormal basis of the span of
$\{\nabla h_{\pi_i}^n(\vec{0}): 1\leq i\leq k-1\}$,
the $k$-th column is $w^+$ and the $(k+1)$-th column is $w^-$.

Consider the function $F:\RR^{(n-1)^2}\to\RR^k$ defined as
\[F_i\left(x_1,\ldots,x_{(n-1)^2}\right)=h^n_{\pi_i}\left(\sum_{j\in [(n-1)^2]}A_{1,j}x_j,\,\ldots,\sum_{j\in [(n-1)^2]}A_{(n-1)^2,j}x_j\right)\]
for $i\in [k]$,
i.e., $F_i$ is the function $h^n_{\pi_i}$ after the unitary transformation given by $A$.
We aim to apply Theorem~\ref{thm:extended} with the variable $x_{(n-1)^2}$ being $z$.
Consider the Jacobian matrix of the function $F$ with respect to the first $(n-1)^2-1$ variables,
i.e., the matrix
\[J=\left(\frac{\partial}{\partial x_j}F_i(\vec 0)\right)_{i\in [k],j\in [(n-1)^2-1]}.\]
Note that the vector $(\alpha_1,\ldots,\alpha_k)^T$ is contained in $\Ker J^T$, and
$w$ can be set to $(\alpha_1,\ldots,\alpha_k)$.
Since each of the last $(n-1)^2-(k-1)$ columns of $A$
is orthogonal to each $\nabla h^n_{\pi_i}(\vec 0)$, $i\in [k]$,
the rank of the first $k-1$ columns of $J$ is $k-1$ and
the last $(n-1)^2-1-(k-1)$ columns of $J$ span the $\Ker J$.
It follows that the function $G:\RR^{(n-1)^2-1-(k-1)}\to\RR$ from the statement of Theorem~\ref{thm:extended}
can be chosen as the restriction of $\alpha_1 F_1+\cdots+\alpha_k F_k$ to the coordinates $x_k,\ldots,x_{(n-1)^2-1}$.
Consider now the Hessian matrix $H$ as in the statement of Theorem~\ref{thm:extended},
i.e., the matrix $H$ is the matrix
\[\left(\frac{\partial}{\partial x_i\partial x_j}G(\vec 0)\right)_{i,j\in [(n-1)^2-1-(k-1)]}.\]
Observe that
\begin{align*}
H_{11} & = (w^+)^T\left(\alpha_1 H^n_{\pi_1}(\vec 0)+\cdots+\alpha_k H^n_{\pi_k}(\vec 0)\right) w^+, \mbox{and}\\
H_{22} & = (w^-)^T\left(\alpha_1 H^n_{\pi_1}(\vec 0)+\cdots+\alpha_k H^n_{\pi_k}(\vec 0)\right) w^-.
\end{align*}
Since the matrix $H$ is symmetric, $H_{11}>0$ and $H_{22}<0$,
it follows that the matrix $H$ has both a positive eigenvalue and a negative eigenvalue.
Hence, the assumptions of Theorem~\ref{thm:extended} are satisfied.

Theorem~\ref{thm:extended} now implies that
there exist $\varepsilon_0>0$ and a function $g:[-\varepsilon_0,\varepsilon_0]\to\RR^{(n-1)^2-1}$ such that
\[F_i(g_1(z),\ldots,g_{(n-1)^2-1}(z),z)=0\]
for every $i\in [k]$ and every $z\in [-\varepsilon_0,\varepsilon_0]$.
We conclude using that $A^{-1}=A^T$ that
\[h^n_{\pi_i}\left(
                  A_{(n-1)^2,1}z+\sum_{j\in [(n-1)^2-1]}A_{j,1}g_j(z),\,\ldots,
                  \,A_{(n-1)^2,(n-1)^2}z+\sum_{j\in [(n-1)^2-1]}A_{j,(n-1)^2}g_j(z)\right)\]
is equal to $0$ for every $i\in [k]$ and every $z\in [-\varepsilon_0,\varepsilon_0]$.
Since $A$ is orthonormal,
the vector
\[\left(A_{(n-1)^2,1}z+\sum_{j\in [(n-1)^2-1]}A_{j,1}g_j(z),\,\ldots,
        \,A_{(n-1)^2,(n-1)^2}z+\sum_{j\in [(n-1)^2-1]}A_{j,(n-1)^2}g_j(z)\right)\]
is zero only if $z=0$.
In particular, this vector for every $z\in [-\varepsilon_0,\varepsilon_0]\setminus\{0\}$
is a non-zero vector for that each $h^n_{\pi_i}$, $i\in [k]$, is zero.
By Lemma~\ref{lem:h=0}, the $k$-tuple $\pi_1,\ldots,\pi_k$ is not quasirandom-forcing.
\end{proof}

Theorem~\ref{thm:non-forcing-kernel} immediately yields the following corollary;
the argument is similar to that used to derive Corollary~\ref{cor:extended} from Theorem~\ref{thm:extended}.

\begin{corollary}
\label{cor:non-forcing}
Let $\pi_1,\ldots,\pi_k$ be a $k$-tuple of permutations such that
$(n-1)^2\ge k+2$ and
$\nabla h^n_{\pi_1}(\vec 0),\ldots,\nabla h^n_{\pi_{k-1}}(\vec 0)$ are linearly independent but
$\alpha_1 \nabla h^n_{\pi_1}(\vec 0)+\cdots+\alpha_k \nabla h^n_{\pi_k}(\vec 0)$ is the zero vector
for non-trivial coefficients $\alpha_1,\ldots,\alpha_k$.
If the matrix $\alpha_1 H^n_{\pi_1}(\vec 0)+\cdots+\alpha_k H^n_{\pi_k}(\vec 0)$
has at least $k$ positive eigenvalues and at least $k$ negative eigenvalues,
then the $k$-tuple $\pi_1,\ldots,\pi_k$ is not quasirandom-forcing.
\end{corollary}

\begin{proof}
It suffices to verify the existence of vectors $w^+$ and $w^-$ as in the statement of Theorem~\ref{thm:non-forcing-kernel}.
By symmetry, it is enough to verify the existence of the vector $w^+$.
Let $\lambda_1,\ldots,\lambda_k$ be positive eigenvalues of the matrix $\alpha_1 H^n_{\pi_1}(\vec 0)+\cdots+\alpha_k H^n_{\pi_k}(\vec 0)$ and
$v_1,\ldots,v_n$ be the corresponding eigenvectors.
Since the vectors $v_1,\ldots,v_k$ are linearly independent,
there exist non-trivial coefficients $\beta_1,\ldots,\beta_k$ such that
the vector $\beta_1v_1+\cdots+\beta_kv_k$ is
orthogonal to each $\nabla h^n_{\pi_i}(\vec 0)$, $i\in [k-1]$.
We now set $w^+$ to $\beta_1v_1+\cdots+\beta_kv_k$ and
observe that the value $(w^+)^T\left(\alpha_1 H^n_{\pi_1}(\vec 0)+\cdots+\alpha_k H^n_{\pi_k}(\vec 0)\right) w^+$
is equal to $\beta_1^2\lambda_1+\cdots+\beta_k^2\lambda_k$,
i.e., it is positive.
\end{proof}

\section{Linearly dependent gradients---specific arguments}
\label{sec:dep}

We now analyze the quadruples $\pi_1,\ldots,\pi_4$ of $4$-point permutations such that
$\nabla h^n_{\pi_i}(\vec 0)$, $i\in [4]$, are linearly dependent.
In this section, we will manage to analyze all such quadruples
except for the quadruple $1234$, $2143$, $3412$, $4321$,
which is analyzed in Section~\ref{sec:special1}, and
the quadruple $1324$, $2413$, $3142$, $4231$,
which is analyzed in Section~\ref{sec:special2}.

The following proposition lists all quadruples of permutations corresponding to the case of the all one matrix
given in the statement of Theorem~\ref{thm:independent}.
The proposition follows from inspecting the list of 576 Latin squares of order 4 and
pruning it down by keeping those Latin squares isomorphic by permuting the elements (this brings the list to the size of 24) and
further pruning it down by rotational and reflection symmetries.
The resulting 12 quadruples of the permutations, which are listed in the proposition,
are visualized in Figure~\ref{fig:one}.

\begin{proposition}
\label{prop:one}
If $\pi_1,\ldots,\pi_4$ is a quadruple of $4$-point permutations such that
$A_{\pi_1}+A_{\pi_2}+A_{\pi_3}+A_{\pi_4}$ is the all one matrix,
then the quadruple is one of the following twelve quadruples up to to the rotational and reflection symmetries:
\begin{itemize}
\item $\pi_1=1234$, $\pi_2=2143$, $\pi_3=3412$, $\pi_4=4321$,
\item $\pi_1=1234$, $\pi_2=2143$, $\pi_3=3421$, $\pi_4=4312$,
\item $\pi_1=1234$, $\pi_2=2341$, $\pi_3=3412$, $\pi_4=4123$,
\item $\pi_1=1234$, $\pi_2=2413$, $\pi_3=3142$, $\pi_4=4321$,
\item $\pi_1=1243$, $\pi_2=2134$, $\pi_3=3421$, $\pi_4=4312$,
\item $\pi_1=1243$, $\pi_2=2314$, $\pi_3=3421$, $\pi_4=4132$,
\item $\pi_1=1324$, $\pi_2=2143$, $\pi_3=3412$, $\pi_4=4231$,
\item $\pi_1=1324$, $\pi_2=2413$, $\pi_3=3142$, $\pi_4=4231$,
\item $\pi_1=1324$, $\pi_2=2413$, $\pi_3=3241$, $\pi_4=4132$,
\item $\pi_1=1342$, $\pi_2=2431$, $\pi_3=3124$, $\pi_4=4213$,
\item $\pi_1=1342$, $\pi_2=2431$, $\pi_3=3214$, $\pi_4=4123$, or
\item $\pi_1=1432$, $\pi_2=2341$, $\pi_3=3214$, $\pi_4=4123$.
\end{itemize}
\end{proposition}

\begin{figure}
\begin{center}
\epsfbox{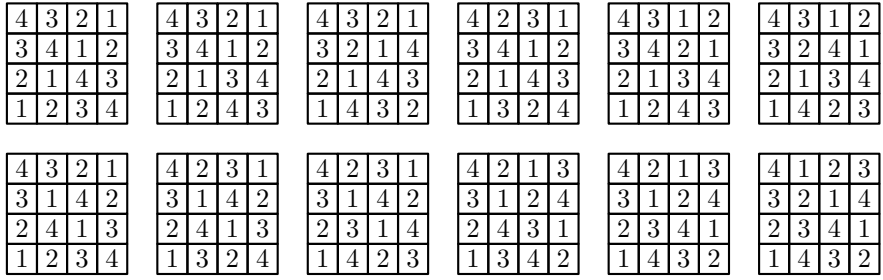}
\end{center}
\caption{Visualization of the twelve quadruples listed in Proposition~\ref{prop:one}.
         Each cell of the grid contains the index of the permutation with its support in that cell.}
\label{fig:one}
\end{figure}

We are now ready to cover the case of the all one matrix given in Theorem~\ref{thm:independent}
with the exception of the two quadruples mentioned at the beginning of this section.

\begin{theorem}
\label{thm:one}
If $\pi_1,\ldots,\pi_4$ is a quadruple of $4$-point permutations such that
$A_{\pi_1}+A_{\pi_2}+A_{\pi_3}+A_{\pi_4}$ is the all one matrix and
the quadruple $\pi_1,\ldots,\pi_4$ is
neither the quadruple $1234$, $2143$, $3412$, $4321$, nor
the quadruple $1324$, $2413$, $3142$, $4231$,
then the set $\{\pi_1,\ldots,\pi_4\}$ is not quasirandom-forcing.
\end{theorem}

\begin{proof}
Fix a quadruple $\pi_1,\ldots,\pi_4$ of $4$-point permutations such that
$A_{\pi_1}+A_{\pi_2}+A_{\pi_3}+A_{\pi_4}$ is the all one matrix.
By symmetry,
we may assume that $\pi_1,\ldots,\pi_4$ is one of the twelve quadruples listed in Proposition~\ref{prop:zero}.
Note that the first and the eighth quadruples are the two exceptional quadruples excluded in the statement.

Assume that $\pi_1,\ldots,\pi_4$ is neither the first, the eighth nor the last quadruple listed in Proposition~\ref{prop:zero}.
If $\pi_1,\ldots,\pi_4$ is the second or tenth, se $n$ to be $7$, and set $n$ to be $5$ otherwise.
In each of these nine cases,
the gradients $\nabla h^n_{\pi_1}(\vec 0)$, $\nabla h^n_{\pi_2}(\vec 0)$, $\nabla h^n_{\pi_3}(\vec 0)$ and $\nabla h^n_{\pi_4}(\vec 0)$
are listed in Appendix and
it can easily checked that
$\nabla h^n_{\pi_1}(\vec 0)$, $\nabla h^n_{\pi_2}(\vec 0)$ and $\nabla h^n_{\pi_3}(\vec 0)$ are linearly independent.
A direct computation yields that
$\nabla h^n_{\pi_1}(\vec 0)+\nabla h^n_{\pi_2}(\vec 0)+\nabla h^n_{\pi_3}(\vec 0)+\nabla h^n_{\pi_4}(\vec 0)$
is equal to the zero vector.
Since the combination $H^n_{\pi_1}(\vec 0)+H^n_{\pi_2}(\vec 0)+H^n_{\pi_3}(\vec 0)+H^n_{\pi_4}(\vec 0)$ of the Hessian matrices
has at least four positive eigenvalues and at least four negative eigenvalues (the numerical values of the eigenvalues and
the $(n-1)\times (n-1)$ matrices are given in Appendix)
unless the quadruple is the last quadruple listed in Proposition~\ref{prop:zero},
the quadruple $\pi_1,\ldots,\pi_4$ is not quasirandom-forcing by Corollary~\ref{cor:non-forcing}.

It remains to analyze the last quadruple listed in Proposition~\ref{prop:zero},
i.e., it is the quadruple $\pi_1=1432$, $\pi_2=2341$, $\pi_3=3214$, $\pi_4=4123$.
Let $A$ be the $9\times 9$ Hessian matrix $H^4_{\pi_1}(\vec 0)+H^4_{\pi_2}(\vec 0)+H^4_{\pi_3}(\vec 0)+H^4_{\pi_4}(\vec 0)$,
which can be found in Appendix.
The gradients of $h^4_{\pi_1}$, $h^4_{\pi_2}$, $h^4_{\pi_3}$ and $h^4_{\pi_4}$ at $\vec 0$ are as follows:
\[
\scalemath{0.75}{
  \begin{array}{cccrrrrrrrrrc}
  \nabla h^4_{\pi_1}(\vec 0) & = & ( 2024/3, & 1160/3, & 392/3, & 1160/3, & 8/3, & -1048/3, & 392/3, & -1048/3, & -1240/3 )^T \\
  \nabla h^4_{\pi_2}(\vec 0) & = & (-392/3, & 1048/3, & 1240/3, & -1160/3, & -8/3, & 1048/3, & -2024/3, & -1160/3, & -392/3 )^T \\
  \nabla h^4_{\pi_3}(\vec 0) & = & ( -1240/3, & -1048/3, & 392/3, & -1048/3, & 8/3, & 1160/3, & 392/3, & 1160/3, & 2024/3 )^T \\
  \nabla h^4_{\pi_4}(\vec 0) & = & ( -392/3, & -1160/3, & -2024/3, & 1048/3, & -8/3, & -1160/3, & 1240/3, & 1048/3, & -392/3 )^T
  \end{array}
}
\]
By Theorem~\ref{thm:non-forcing-kernel},
it suffices to show that there exist vectors $w^+$ and $w^-$ that
are orthogonal to $\nabla h^4_{\pi_1}(0)$, $\nabla h^4_{\pi_2}(0)$, $\nabla h^4_{\pi_3}(0)$ and $\nabla h^4_{\pi_4}(0)$ such that
$(w^+)^T Aw>0$ and $(w^-)^T Aw<0$.
Since the matrix $A$ has eight positive eigenvalues,
the vector $w^+$ can be chosen as a non-trivial linear combination of the eigenvectors associated with any four such eigenvalues.
Next choose $w^-$ as follows:
\[ w^- = (-23, 42, -23, 128, 112, 128, 0, 8, 0)^T.\]
It is straightforward to verify $w^-$
is orthogonal to $\nabla h^4_{\pi_1}(0)$, $\nabla h^4_{\pi_1}(2)$, $\nabla h^4_{\pi_1}(3)$ and $\nabla h^4_{\pi_1}(4)$ and
$(w^-)^T Aw^-=-115456$.
Theorem~\ref{thm:non-forcing-kernel} now implies that
the quadruple $\pi_1=1432$, $\pi_2=2341$, $\pi_3=3214$, $\pi_4=4123$ is not quasirandom-forcing.
\end{proof}

We remark that
it seems that Theorem~\ref{thm:non-forcing-kernel} cannot be used
in the case of the first and eighth quadruple in the statement of Theorem~\ref{thm:one}. 
In the case of the quadruple $\pi_1=1234$, $\pi_2=2143$, $\pi_3=3412$, $\pi_4=4321$,
the $(n-1)\times (n-1)$ Hessian matrix $H^n_{\pi_1}(\vec 0)+H^n_{\pi_2}(\vec 0)+H^n_{\pi_3}(\vec 0)+H^n_{\pi_4}(\vec 0)$
is positive definite for every $n\in\{3,4,5,6,7,8\}$.
In the case of the quadruple $\pi_1=1324$, $\pi_2=2413$, $\pi_3=3142$, $\pi_4=4231$,
the $(n-1)\times (n-1)$ Hessian matrix $H^n_{\pi_1}(\vec 0)+H^n_{\pi_2}(\vec 0)+H^n_{\pi_3}(\vec 0)+H^n_{\pi_4}(\vec 0)$
has a single positive eigenvalue for every $n\in\{5,6,7,8\}$,
however, $w^T(H^n_{\pi_1}(\vec 0)+H^n_{\pi_2}(\vec 0)+H^n_{\pi_3}(\vec 0)+H^n_{\pi_4}(\vec 0))w<0$
for every vector $w$ orthogonal to the gradients $\nabla h^n_{\pi_1}(\vec 0)$, $\nabla h^n_{\pi_2}(\vec 0)$, $\nabla h^n_{\pi_3}(\vec 0)$ and $\nabla h^n_{\pi_4}(\vec 0)$,
i.e.,
the restriction of $H^n_{\pi_1}(\vec 0)+H^n_{\pi_2}(\vec 0)+H^n_{\pi_3}(\vec 0)+H^n_{\pi_4}(\vec 0)$ to vectors orthogonal to the gradients
is negative definite.
Hence, these two cases need to be analyzed by specific arguments presented in the next two sections. 

We next consider the second case given in Theorem~\ref{thm:independent}.
The following proposition lists all quadruples of permutations corresponding to the case of the zero matrix
given in the statement of Theorem~\ref{thm:independent}.
The listed quadruples of the permutations are visualized in Figure~\ref{fig:zero}.

\begin{proposition}
\label{prop:zero}
If $\pi_1,\ldots,\pi_4$ is a quadruple of $4$-point permutations such that
$A_{\pi_1}+A_{\pi_2}-A_{\pi_3}-A_{\pi_4}$ is the zero matrix,
then the quadruple is one of the following seven quadruples up to to the rotational and reflection symmetries:
\begin{itemize}
\item $\pi_1=1234$, $\pi_2=2143$, $\pi_3=1243$, $\pi_4=2134$,
\item $\pi_1=1234$, $\pi_2=3412$, $\pi_3=1432$, $\pi_4=3214$,
\item $\pi_1=1234$, $\pi_2=4321$, $\pi_3=1324$, $\pi_4=4231$,
\item $\pi_1=1243$, $\pi_2=3421$, $\pi_3=1423$, $\pi_4=3241$,
\item $\pi_1=1324$, $\pi_2=2413$, $\pi_3=1423$, $\pi_4=2314$,
\item $\pi_1=1342$, $\pi_2=2431$, $\pi_3=1432$, $\pi_4=2341$, or
\item $\pi_1=2143$, $\pi_2=3412$, $\pi_3=2413$, $\pi_4=3142$.
\end{itemize}
\end{proposition}

\begin{figure}
\begin{center}
\epsfbox{qperm4-1.mps}
\end{center}
\caption{Visualization of the seven quadruples listed in Proposition~\ref{prop:zero}.
         Each cell of the grid contains the indices of the permutations with their support in that cell.}
\label{fig:zero}
\end{figure}

\begin{proof}
As argued in the last paragraph of the proof of Theorem~\ref{thm:independent},
we can assume that
there exists $i\in [4]$ such that $\pi_1(i)=\pi_3(i)\not=\pi_2(i)$ and $\pi_2(i)=\pi_4(i)$ and
$j\in [4]$ such that $\pi_1(j)=\pi_4(j)\not=\pi_2(j)$ and $\pi_2(j)=\pi_3(j)$.
Let $i'$ be such that $\pi_2(i')=\pi_1(i)$, and
observe that $\pi_4(i')=\pi_2(i')\not=\pi_1(i')$ and $\pi_1(i')=\pi_3(i')$.
Similarly,
let $j'$ be such that $\pi_2(j')=\pi_1(j)$ and
observe that $\pi_3(j')=\pi_2(j')\not=\pi_1(j')$ and $\pi_1(j')=\pi_4(j')$.
Note that the values of $i$, $i'$ and $\pi_1(i)$ and $\pi_1(i')$ fully determine
the permutations $\pi_2$, $\pi_3$ and $\pi_4$ up to a possible swap of the permutations $\pi_3$ and $\pi_4$.

Next observe that we can assume that
$\{i,i'\}$ is one of the sets $\{1,2\}$, $\{1,3\}$, $\{1,4\}$ (in the remaining cases,
we swap the roles of $\pi_1$ and $\pi_2$) and,
using the reflection symmetry, that
$\{\pi_1(i),\pi_1(i')\}$ is one of the sets $\{1,2\}$, $\{1,3\}$, $\{1,4\}$ and $\{2,3\}$.
This results in $12$ possible quadruples of permutations $\pi_1$, $\pi_2$, $\pi_3$ and $\pi_4$,
which after the inspection for rotational symmetries yields the eight quadruples listed in the statement of the proposition.
\end{proof}

We conclude this section with the following theorem,
which covers the second case given in Theorem~\ref{thm:independent}.

\begin{theorem}
\label{thm:zero}
If $\pi_1,\ldots,\pi_4$ is a quadruple of $4$-point permutations such that
$A_{\pi_1}+A_{\pi_2}-A_{\pi_3}-A_{\pi_4}$ is the zero matrix,
then the set $\{\pi_1,\ldots,\pi_4\}$ is not quasirandom-forcing.
\end{theorem}

\begin{proof}
Fix a quadruple $\pi_1,\ldots,\pi_4$ of $4$-point permutations such that
$A_{\pi_1}+A_{\pi_2}-A_{\pi_3}-A_{\pi_4}$ is the zero matrix.
By symmetry,
we may assume that $\pi_1,\ldots,\pi_4$ is one of the seven quadruples listed in Proposition~\ref{prop:zero}.
If it is the third or the seventh quadruple listed in the proposition, set $n$ to be $7$;
otherwise, set $n$ to be $5$.
The gradients $\nabla h^n_{\pi_1}(\vec 0)$, $\nabla h^n_{\pi_2}(\vec 0)$, $\nabla h^n_{\pi_3}(\vec 0)$ and $\nabla h^n_{\pi_4}(\vec 0)$
are listed in Appendix and
it can easily checked that
$\nabla h^n_{\pi_1}(\vec 0)$, $\nabla h^n_{\pi_2}(\vec 0)$ and $\nabla h^n_{\pi_3}(\vec 0)$ are linearly independent.
A direct computation yields that
the gradient $h^n_{\pi_1}(\vec 0)+h^n_{\pi_2}(\vec 0)-h^n_{\pi_3}(\vec 0)-h^n_{\pi_4}(\vec 0)$ is equal to zero.
Since the combination $H^n_{\pi_1}(\vec 0)+H^n_{\pi_2}(\vec 0)-H^n_{\pi_3}(\vec 0)-H^n_{\pi_4}(\vec 0)$ of the Hessian matrices
has at least four positive eigenvalues and at least four negative eigenvalues (the numerical values of the eigenvalues and
the $(n-1)\times (n-1)$ matrices are given in Appendix),
the quadruple $\pi_1,\ldots,\pi_4$ is not quasirandom-forcing by Corollary~\ref{cor:non-forcing}.
\end{proof}

\section{First exceptional quadruple}
\label{sec:special1}

Our goal in this section is to prove the following theorem.

\begin{theorem}
\label{thm:special1}
The quadruple $1234,2143,3412,4321$ is not quasirandom-forcing. 
\end{theorem}

Consider the following three $32$-point permutations $\pi_1$, $\pi_2$ and $\pi_3$; here, and in the next section, we separate the values of the permutations by commas for clarity.
\[
\scalemath{0.85}{
  \begin{array}{ccl}
  \pi_1 & = & 31,1,29,3,27,5,25,7,23,9,21,11,19,13,17,15,18,16,20,14,22,12,24,10,26,8,28,6,30,4,32,2\\
  \pi_2 & = & 17,15,19,13,21,11,23,9,25,7,27,5,29,3,31,1,32,2,30,4,28,6,26,8,24,10,22,12,20,14,18,16\\
  \pi_3 & = & 12,27,13,26,10,2,4,16,15,28,14,32,30,22,24,25,8,9,11,3,1,19,5,18,17,29,31,23,7,20,6,21
  \end{array}
  }
\]
We define the permuton $\mu_{x,y,z}$ for non-negative reals $x$, $y$ and $z$ such that $x+y+z=1$ as
\[\mu_{x,y,z}=\mu\left[xA_{\pi_1}+yA_{\pi_2}+zA_{\pi_3}\right].\]
The permuton $\mu_{x,y,z}$ is illustrated in Figure~\ref{fig:muxyz}.
Since the permuton $\mu_{x,y,z}$ is rotationally symmetric,
it holds that
\[d\left(1234,\mu_{x,y,z}\right)=d\left(4321,\mu_{x,y,z}\right)\qquad\mbox{and}\qquad
  d\left(2143,\mu_{x,y,z}\right)=d\left(3412,\mu_{x,y,z}\right).\]
We determined the following expressions for $d\left(1234,\mu_{x,y,z}\right)$ and $d\left(2143,\mu_{x,y,z}\right)$ using a computer:
\begin{align*}
d\left(1234,\mu_{x,y,z}\right)&=\left(\frac{1}{603979776}\right)\cdot\bigl(24839424z^4+82000896yz^3+98446080y^2z^2\\&+58334208y^3z+17049600y^4+115072512xz^3+286929408xyz^2\\&+223108608xy^2z+59338752xy^3+207333120x^2z^2+348275712x^2yz\\&+133929984x^2y^2+183501312x^3z+158042112x^3y+66401280x^4\bigr),\mbox{ and}\\
d\left(2143,\mu_{x,y,z}\right)&=\left(\frac{1}{603979776}\right)\cdot\bigl(24870144z^4+124873728yz^3+223722240y^2z^2\\&+178179072y^3z+54460416y^4+88207872xz^3+304359936xyz^2\\&+357315072xy^2z+147038208xy^3+101564160x^2z^2+222084096x^2yz\\&+135416832x^2y^2+42948096x^3z+47560704x^3y+4721664x^4\bigr).
\end{align*}

\begin{figure}
\begin{center}
\epsfbox{qperm4-0.mps}
\end{center}
\caption{Visualization of the permuton $\mu_{x,y,z}$ defined in Section~\ref{sec:special1}.}
\label{fig:muxyz}
\end{figure}

We will investigate the densities of the permutations $1234$ and $2143$ in $\mu_{x,y,z}$
for $x=st$, $y=s(1-t)$ and $z=1-s$ where $s,t\in [0,1]$.
To do so, we define
\[g_1(s,t):=d\left(1234,\mu_{st,s(1-t),1-s}\right)\quad\mbox{and}\quad
  g_2(s,t):=d\left(2143,\mu_{st,s(1-t),1-s}\right).\]
A substitution into the expressions for $d\left(1234,\mu_{x,y,z}\right)$ and $d\left(2143,\mu_{x,y,z}\right)$ given above yields that
\begin{align*}
g_1(s,t) & = \left(\frac{1}{786432}\right)\cdot\bigl(51252 s^3 t^2 - 42648 s^3 t + 10530 s^3 + 24544 s^2 t^2- 11950 s^2 t + 1927 s^2 \\
         & + 43062 s t - 22600 s + 32343\bigr),\mbox{ and}\\
g_2(s,t) & = \left(\frac{1}{786432}\right)\cdot\bigl(180 s^3 t^2 - 1368 s^3 t + 7650 s^3 + 27248 s^2 t^2 - 43082 s^2 t - 2185 s^2 \\
         & - 47742 s t + 33064 s + 32383\bigr).
\end{align*}
Our aim is to show that there exist $s,t\in [0,1]$ such that $g_1(s,t)=g_2(s,t)=1/24$,
which will be established in Lemma~\ref{lem:st}.

\begin{figure}
\begin{center}
\epsfbox{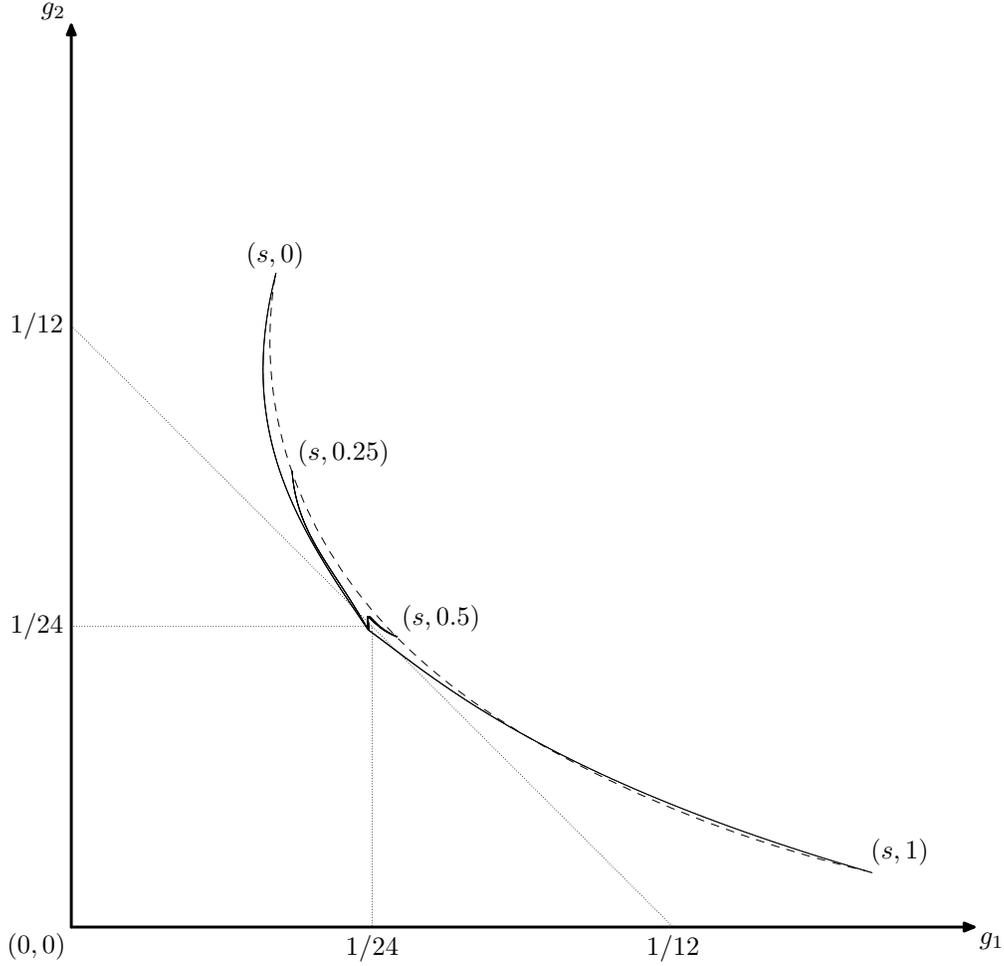}
\end{center}
\caption{A plot of the curves $(g_1(s,t),g_2(s,t))$ for $(s,t)=(s,0)$, $(s,t)=(s,1/4)$, $(s,t)=(s,1/2)$ and $(s,t)=(s,1)$.
         In addition, the curve for $(s,t)=(1,t)$ is drawn as dashed.}
\label{fig:curveplot}
\end{figure}

\begin{figure}
\begin{center}
\epsfbox{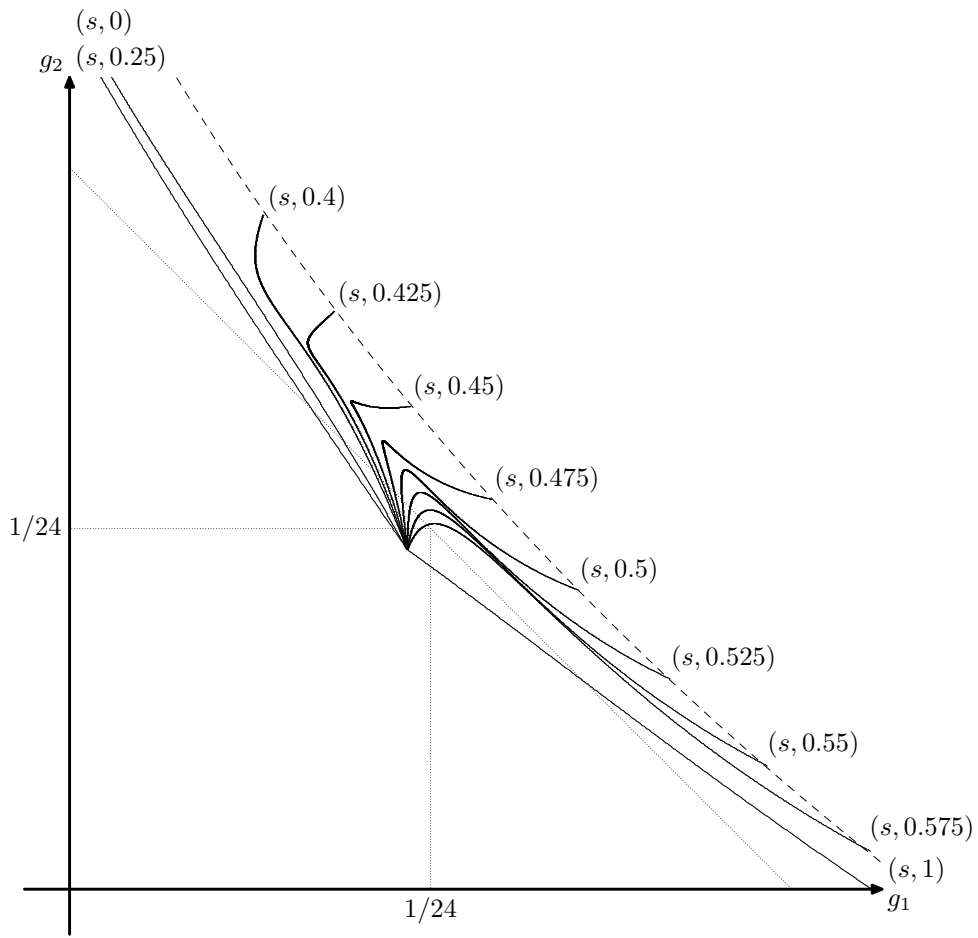}
\end{center}
\caption{A detail of the plot from Figure~\ref{fig:curveplot} around the point $(1/24,1/24)$.}
\label{fig:detail}
\end{figure}

As the first step,
we show that $g_1(s,t)+g_2(s,t)$ is an increasing function of $s$ for any fixed $t\in [0,1]$.
Also see Figures~\ref{fig:curveplot} and~\ref{fig:detail} for illustration.
Note that $g_1(0,t)=g_1(0,t')$ and $g_2(0,t)=g_2(0,t')$ for all $t,t'\in [0,1]$.

\begin{lemma}
\label{lem:inc}
It holds that $\frac{\partial}{\partial s}(g_1(s,t)+g_2(s,t))>0$ for all $s,t\in[0,1]$. 
\end{lemma}

\begin{proof}
A direct computation yields that
\begin{equation}
\scalemath{0.95}{
\frac{\partial}{\partial s}(g_1(s,t)+g_2(s,t)) =
\frac{\left(12858 t^2 - 11004 t + 4545\right)s^2 + \left(8632 t^2 - 9172 t - 43 \right)s - 390 t + 872}{65536}.}
\label{eq:inc}
\end{equation}
Note that when $t\in [0,1]$ is fixed, we can view the numerator in \eqref{eq:inc} as a quadratic function in $s$.
We show for any fixed $t\in [0,1]$ that
the numerator viewed as a quadratic function in $s$ has no root.
Since the numerator is positive for $s=0$ (specifically, it is equal to $872-390t>0$),
this would yield that the function is positive for all $s\in [0,1]$ and so the lemma would follow.

Let $D(t)$ be the discriminant of the quadratic function in the numerator in \eqref{eq:inc},
i.e.,
\begin{align*}
D(t) & = \left(8632 t^2 - 9172 t - 43 \right)^2 - 4\left(12858 t^2 - 11004 t + 4545\right)\left( - 390 t + 872\right)\\
     & = 74511424 t^4 - 138286928 t^3 + 21368288 t^2 + 46260944 t - 15851111.
\end{align*}
The quadratic function in the numerator in \eqref{eq:inc} has no root if and only if $D(t)<0$, and
we will indeed show that $D(t)<0$ for every $t\in [0,1]$.
Note that the derivative of $D(t)$ is
\[\frac{\dd D}{\dd t}(t)=298045696 t^3  - 414860784 t^2 + 42736576 t + 46260944.\]
We next evaluate the derivative of $D(t)$ for some specific values of $t$:
\[
\begin{array}{lll}
\frac{\dd D}{\dd t}(-1) & = -709382112 & < 0,\\[2pt]
\frac{\dd D}{\dd t}(0) & = 46260944 & > 0,\\[2pt]
\frac{\dd D}{\dd t}(0.50788) & = \frac{122688081741911}{122070312500} & > 0,\\[2pt]
\frac{\dd D}{\dd t}(0.50789) & = -\frac{1856263737032121}{3906250000000} & < 0,\\[2pt]
\frac{\dd D}{\dd t}(1) & = -27817568 & < 0,\\[2pt]
\frac{\dd D}{\dd t}(2) & = 856656528 & > 0.
\end{array}
\]
Since the derivative of $D(t)$ is a cubic function in $t$, it has at most three roots.
The evaluations above yield that
one of the roots is in the interval $(-1,0)$, one in the interval $(0.50788,0.50789)$ and one in the interval $(1,2)$.
Let $t_0$ be the unique root of the derivative of $D(t)$ in the interval $[0,1]$;
note that $t_0\in (0.50788,0.50789)$.

Observe that the function $D(t)$ is increasing for $t\in [0,t_0]$ and decreasing for $t\in [t_0,1]$.
We next estimate the value of $D(t_0)$ using that $t_0\in (0.50788,0.50789)$:
\begin{align*}
D(t_0) & = 74511424 t_0^4 - 138286928 t_0^3 + 21368288 t_0^2 + 46260944 t_0 - 15851111 \\
       & \leq 74511424(0.50789)^4 - 138286928 (0.50788)^3 + 21368288 (0.50789)^2 \\
       & + 46260944 (0.50789) - 15851111\\
       & = -\frac{2795347632052487974719}{1562500000000000000}<0.
\end{align*}       
It follows that the value of $D(t)$ is negative for all $t\in [0,1]$ and
so the fraction in \eqref{eq:inc} is positive for all $s,t\in [0,1]$.
\end{proof}

We next show that for each $t\in [0,1]$,
the curve $(g_1(s,t),g_2(s,t))$ parameterized by $s\in [0,1]$
intersects the line $(x,y)$ with $x+y=1/12$ at a single point.

\begin{lemma}
\label{lem:sum1/12}
For every $t\in[0,1]$, there exists a unique $s\in[0,1]$ such that $g_1(s,t)+g_2(s,t) = 1/12$.
\end{lemma}

\begin{proof}
Fix $t\in [0,1]$.
Note that $g_1(0,t)+g_2(0,t)=\frac{32343}{786432} + \frac{32383}{786432}=\frac{32363}{393216} < \frac{1}{12}$. 
On the other hand, it holds that
\begin{align*}
g_1(1,t)+g_1(1,t) & = \frac{4301t^2-4322t+1149}{32768}+\frac{1}{12} \\
                  & \ge \frac{4301t^2-4300t+1127}{32768}+\frac{1}{12} \\
		  & = \frac{t^2+4300(t-1/2)^2+52}{32768}+\frac{1}{12} > \frac{1}{12}.
\end{align*}
Therefore, the Intermediate Value Theorem yields that
there exists $s\in [0,1]$ such that $g_1(s,t)+g_2(s,t)=1/12$. 
The uniqueness follows from Lemma~\ref{lem:inc},
which yields that the sum $g_1(s,t)+g_2(s,t)$ is an increasing function of $s$ for any fixed $t\in [0,1]$.
\end{proof}

We are now ready to prove the existence of $s$ and $t$ such that $g_1(s,t)=g_2(s,t)=1/24$.

\begin{lemma}
\label{lem:st}
There exist $s,t\in[0,1]$ such that $g_1(s,t)=g_2(s,t)=1/24$.
\end{lemma}

\begin{proof}
For $t\in [0,1]$,
let $s(t)$ be the unique element of $[0,1]$ such that $g_1(s(t),t)+g_2(s(t),t)=1/12$,
which exists by Lemma~\ref{lem:sum1/12}.
Define a function $h:[0,1]\to [0,1]$ as $h(t)=g_1(s(t),t)$.
The uniqueness of $s(t)$ and the continuity of the functions $g_1$ and $g_2$ implies that
$s(t)$ viewed as function from $[0,1]$ to $[0,1]$ is continuous.
Hence, the function $h(t)$ is continuous as well.
Our aim is to show that there exists $t_0$ such that $h(t_0)=1/24$.

We next show that $h(0)<1/24$ and $h(1)>1/24$.
The substitution $t=0$ in the expressions for $g_1(s,t)$ and $g_2(s,t)$ yields that
\begin{align*}
g_1(s,0) & = \frac{1}{786432}\cdot\left(10530 s^3 + 1927 s^2  - 22600 s + 32343\right) \mbox{ and}\\
g_2(s,0) & = \frac{1}{786432}\cdot\left(7650 s^3 - 2185 s^2+ 33064 s + 32383\right)
\end{align*}
for every $s\in [0,1]$.
Since it holds that $g_1(7/10,0)+g_2(7/10,0)=\frac{1954003}{19660800}>\frac{1}{12}$ and
$g_1(s,0)+g_2(s,0)$ is an increasing function of $s$ by Lemma~\ref{lem:inc},
we obtain that $s(0)<7/10$.
A direct computation yields that
\[
\frac{\partial}{\partial s}g_1(s,0)=\frac{1}{786432}\cdot\left(31590 s^2 + 3854 s  - 22600\right).
\]
Since the partial derivative is a quadratic function in $s$ that is negative for both $s=0$ and $s=7/10$,
the partial derivative is negative for all $s\in [0,7/10]$.
It follows that $g_1(s,0)$ is a decreasing function of $s$ for $s\in [0,7/10]$ and
so $h(0)=g_1(s(0),0)<g_1(0,0)=\frac{32343}{786432}<1/24$.

We next substitute $t=1$ in the expressions for $g_1(s,t)$ and $g_2(s,t)$ and obtain that
\begin{align*}
g_1(s,1) & = \frac{1}{786432}\cdot\left(19134 s^3 + 14521 s^2 + 20462 s + 32343\right) \mbox{ and}\\
g_2(s,1) & = \frac{1}{786432}\cdot\left(6462 s^3 - 18019 s^2 - 14678 s + 32383\right)
\end{align*}
for every $s\in [0,1]$.
Since it holds that $g_1(1/10,1)+g_2(1/10,1)=\frac{8161877}{98304000}<\frac{1}{12}$ and
the function $g_1(s,0)+g_2(s,0)$ is increasing in $s$ by Lemma~\ref{lem:inc},
we obtain that $s(1)\in [1/10,1]$.
Note that $g(1/10,1)=\frac{1439731}{32768000}>1/24$.
A direct computation yields that
\[
\frac{\partial}{\partial s}g_1(s,1)=\frac{1}{786432}\cdot\left(57402 s^2 + 29042 s + 20462\right),
\]
which is positive for all $s\in [0,1]$,
i.e., $g_1(s,1)$ is an increasing function of $s$.
It follows that $h(1)=g_1(s(1),1)>g_1(1/10,1)>1/24$.

We have shown that $h(0)<1/24$ and $h(1)>1/24$.
Since $h$ is a continuous function,
the Intermediate Value Theorem implies that there exists $t_0\in [0,1]$ such that $h(t_0)=1/24$.
It follows that $g_1(s(t_0),t_0)=h(t_0)=1/24$ and $g_2(s(t_0),t_0)=1/24$,
which yields the existence of $s$ and $t$ claimed in the statement of the lemma.
\end{proof}

Lemma~\ref{lem:st} readily implies Theorem~\ref{thm:special1}.

\begin{proof}[Proof of Theorem~\ref{thm:special1}]
Let $s$ and $t$ be as in Lemma~\ref{lem:st} and consider the permuton $\mu:=\mu_{st,s(1-t),1-s}$.
Note that 
\[d\left(4321,\mu\right)=d\left(1234,\mu\right)=g_1(s,t)=1/24\mbox{ and }
  d\left(3412,\mu\right)=d\left(2143,\mu\right)=g_2(s,t)=1/24.\]
Since $\mu$ is not the uniform measure on $[0,1]^2$,
the quadruple $1234,2143,3412,4321$ is not quasirandom-forcing.
\end{proof}

\section{Second exceptional quadruple}
\label{sec:special2}

In this section, we prove the following theorem.

\begin{theorem}
\label{thm:special2}
The quadruple $1324,2413,3142,4231$ is not quasirandom-forcing. 
\end{theorem}

Our approach is similar to that of the previous section. For each $x,y,z\in[0,1]$ with $x+y+z=1$, we let $\mu_{x,y,z}=\mu\left[xA_{\pi_1}+yA_{\pi_2}+zA_{\pi_3}\right]$ where $\pi_1,\pi_2$ and $\pi_3$ are the following three $20$-point permutations.
\[\begin{array}{ccl}
  \pi_1 & = & 19,1,17,3,15,5,13,7,11,9,12,10,14,8,16,6,18,4,20,2\\
  \pi_2 & = & 10,12,8,14,6,16,4,18,2,20,1,19,3,17,5,15,7,13,9,11\\
  \pi_3 & = & 5,4,15,19,20,3,10,9,13,14,7,8,12,11,18,1,2,6,17,16
  \end{array}
\]
See Figure~\ref{fig:muxyz2} for an illustration of the permuton $\mu_{x,y,z}$.

\begin{figure}
\begin{center}
\epsfxsize=14cm
\epsfbox{qperm4-4.mps}
\end{center}
\caption{The permuton $\mu_{x,y,z}$. The measure of each unlabeled square is zero and the measure of each labelled square is equal to the label of the square multiplied by $1/20$. }
\label{fig:muxyz2}
\end{figure}

As in Section~\ref{sec:special1}, $\mu_{x,y,z}$ is rotationally symmetric and so $d\left(4231,\mu_{x,y,z}\right)=d\left(1324,\mu_{x,y,z}\right)$ and $d\left(2413,\mu_{x,y,z}\right)=d\left(3142,\mu_{x,y,z}\right)$. Using a computer, we have determined that
\begin{align*}
d\left(1324,\mu_{x,y,z}\right)&=\left(\frac{1}{92160000}\right)\cdot\bigl(4835808z^4+14871552yz^3+17961600y^2z^2\\&+8318976y^3z+488160y^4+20615424xz^3+51350016xyz^2\\&+42776832xy^2z+9885696xy^3+33465216x^2z^2+58404864x^2yz\\&+24670272x^2y^2+23476992x^3z+21083136x^3y+5810400x^4\bigr),\mbox{ and}\\
d\left(2413,\mu_{x,y,z}\right)&=\left(\frac{1}{92160000}\right)\cdot\bigl(4092384z^4+21557760yz^3+36739968y^2z^2\\&+24916992y^3z+5737440y^4+10719744xz^3+46809600xyz^2\\&+54297600xy^2z+18124800xy^3+9955968x^2z^2+32884224x^2yz\\&+19314240x^2y^2+3310080x^3z+6927360x^3y+480x^4\bigr).
\end{align*}
We investigate the densities of the permutations $1324$ and $2413$ in $\mu_{x,y,z}$
for $x=st$, $y=s(1-t)$ and $z=1-s$ where $s,t\in [0,1]$. Define $g_1(s,t):=d\left(1324,\mu_{st,s(1-t),t}\right)$ and $g_2(s,t):=d\left(2413,\mu_{st,s(1-t),t}\right)$. By substituting, we get
\begin{align*}
g_1(s,t)&=\left(\frac{1}{960000}\right)\cdot\bigl(-864s^4 t^3 + 2196s^4 t^2 - 2124s^4 t +990s^4 -4896s^3 t^3 \\
        & - 24432s^3 t^2 + 43728s^3 t - 24300s^3 + 800s^2 t^2 - 18800s^2 t + 24602s^2 + 59832st \\
	& -46580s + 50373\bigr),\mbox{ and}\\
g_2(s,t)&=\left(\frac{1}{960000}\right)\cdot\bigl(-864 s^4 t^3 + 2196 s^4 t^2 - 2124 s^4 t + 990 s^4 - 2016 s^3 t^3 \\
        & - 7632 s^3 t^2 + 3888 s^3 t - 2700 s^3 - 1184 s^2 t^2 + 60872 s^2 t - 35198 s^2 - 112896 s t \\
	& + 54044 s + 42629\bigr).
\end{align*}
As in Section~\ref{sec:special1},
we aim at finding $s$ and $t$ such that $g_1(s,t)=g_2(s,t)=1/24$, which will be achieved in Lemma~\ref{lem:st2}.

\begin{figure}
\begin{center}
\epsfbox{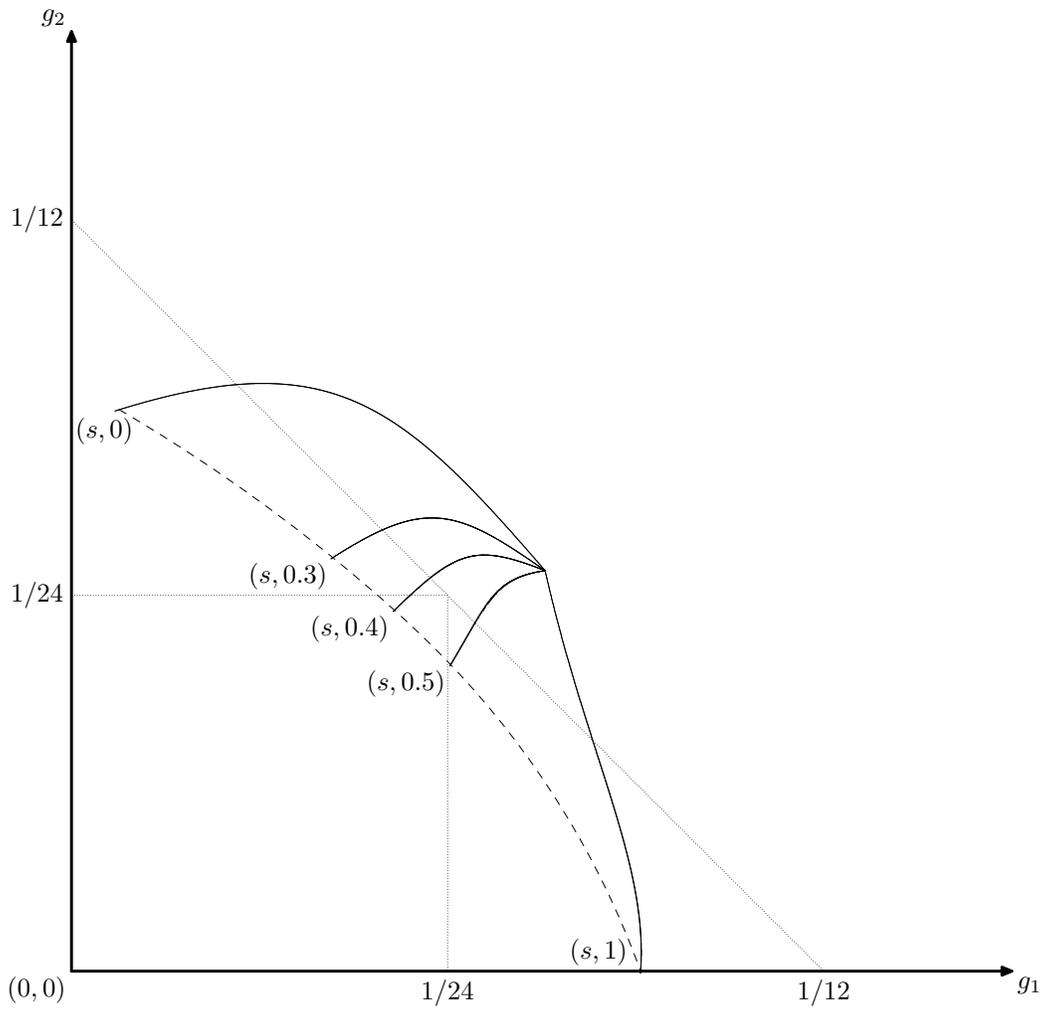}
\end{center}
\caption{A plot of the curves $(g_1(s,t),g_2(s,t))$ for $(s,t)=(s,0)$, $(s,t)=(s,0.3)$, $(s,t)=(s,0.4)$, $(s,t)=(s,0.5)$ and $(s,t)=(s,1)$.
         In addition, the curve for $(s,t)=(1,t)$ is drawn as dashed.}
\label{fig:curveplot2}
\end{figure}

We start by showing that $g_1(s,t)+g_2(s,t)$ is an decreasing function of $s$ for any fixed $t\in [0.15,1]$.
See Figure~\ref{fig:curveplot2} for illustration.
Note that $g_1(0,t)=g_1(0,t')$ and $g_2(0,t)=g_2(0,t')$ for all $t,t'\in [0,1]$.

\begin{lemma}
\label{lem:dec}
For any fixed $t\in[0.15,1]$, we have $\frac{\partial}{\partial s}(g_1(s,t)+g_2(s,t))<0$ for all $s\in[0,1]$. 
\end{lemma}

\begin{proof}
We compute
\begin{equation}
\begin{aligned}
\frac{\partial}{\partial s}(g_1(s,t)+g_2(s,t)) &= \left(\frac{1}{40000}\right)\cdot\bigl(\left(-288 t^3 + 732 t^2 - 708 t + 330\right)s^3\\& +\left(-864 t^3 - 4008 t^2 + 5952 t  - 3375 \right)s^2\\& + \left(-32 t^2 + 3506 t - 883 \right)s - 2211 t + 311 \bigr).
\label{eq:decfirst}
\end{aligned}
\end{equation}
When $t$ is fixed, the function \eqref{eq:decfirst} is a cubic function in $s$. First, let us analyze the coefficient of $s^3$ in the above expression, namely $f_3(t)=\frac{-288 t^3 + 732 t^2 - 708 t + 330}{40000}$.
The discriminant of the equation $f_3(t)=0$ is equal to 
\begin{align*}
\frac{1}{40000^4} & \left((732)^2(708)^2 - 4\cdot(-288)(-708)^3 - 4\cdot(732)^3\cdot 330 \right.\\
                  & \left.-27\cdot(-288)^2(330)^2+ 18\cdot 288\cdot 732\cdot 708\cdot 330\right),
\end{align*}
which is negative.
Consequently, the equation $f_3(t)=0$ has a unique real root.
Since it holds that $f_3(1)=\frac{33}{40000}>0$ and $f_3(2)=\frac{-231}{20000}<0$,
the unique root of $f_3(t)=0$ is in the interval $(1,2)$ and so, in particular, $f_3(t)>0$ for all $t\in[0.15,1]$.
It follows that $f_3(t) s^3 \leq f_3(t) s^2$ for all $s\in [0,1]$.
Combining this estimate with \eqref{eq:decfirst},
we can bound $\frac{\partial}{\partial s}(g_1(s,t)+g_2(s,t))$ as follows:
\begin{equation}
\begin{aligned}
\frac{\partial}{\partial s}(g_1(s,t)+g_2(s,t)) &< \left(\frac{1}{40000}\right)\cdot\left(\left(- 1152 t^3- 2946 t^2 + 5244 t-3375\right)s^2\right.\\
                                               & \left. + \left(-32 t^2 + 3506 t - 883 \right)s - 2211 t + 311 \right). 
\end{aligned}
\label{eq:decsecond}
\end{equation}
The polynomial in the parentheses in the right side of \eqref{eq:decsecond} will be denoted by $b(s,t)$.
To establish the lemma, it is enough to show that $b(s,t)<0$ for all $s\in[0,1]$ and $t\in[0.15,1]$.
Note that $b(0,t)=-2211t+311$ is clearly negative for all $t\in [0.15,1]$, and
so it suffices to show that the discriminant of the equation $b(s,t)=0$, when viewed as a quadratic equation in $s$,
is negative for any fixed $t\in[0.15,1]$. 

Let $D(t)$ be the discriminant of the equation $b(s,t)=0$ when $t\in [0.15,1]$ is fixed,
i.e., $b(s,t)=0$ is a quadratic equation in $s$.
It follows that
\begin{align*}
D(t)&=\left(-32 t^2 + 3506 t - 883 \right)^2 - 4\left(- 1152 t^3- 2946 t^2 + 5244 t-3375\right) \left(-2211 t +311  \right) \\ 
&= - 10187264 t^4 - 24845720 t^3 + 62391308 t^2 - 42563632 t + 4978189.
\end{align*}
The derivative of $D(t)$ when viewed as a function of $t$ is
\[\frac{\dd D}{\dd t}(t)= - 40749056 t^3 - 74537160 t^2 + 124782616 t-42563632.\]
In particular,
the derivative of $D(t)$ is a cubic function in $t$;
since the discriminant of the equation $D(t)=0$ is negative,
the equation $D(t)=0$ has a unique real root.
Since it holds that $\frac{dD}{dt}(-3)>0$ and $\frac{dD}{dt}(0)<0$, 
the unique root of the equation $D(t)=0$ belongs to the interval $(-3,0)$.
It follows that $\frac{\dd D}{\dd t}(t)$ is negative for all $t\in[0.15,1]$.
Finally, since it holds that $D(0.15)=-91563<0$,
we obtain that $D(t)$ is negative for all $t\in[0.15,1]$, and
since it holds that $b(0,t)<0$ for all $t\in [0.15,1]$,
we conclude that $b(s,t)<0$ for all $s\in[0,1]$ and $t\in[0.15,1]$.
This completes the proof of the lemma.
\end{proof}

Next, we show that for each $t\in[0.15,1]$, there is a unique $s\in[0,1]$ such that $g_1(s,t)+g_2(s,t)=1/12$. 

\begin{lemma}
\label{lem:sum1/12second}
For each $t\in[0.15,1]$, there is a unique $s\in[0,1]$ such that $g_1(s,t)+g_2(s,t) = 1/12$.
\end{lemma}

\begin{proof}
For any $t\in [0.15,1]$,
the sum $g_1(0,t)+g_2(0,t)$ is equal to $\frac{50373}{960000} + \frac{42629}{960000}=\frac{46501}{480000}>\frac{1}{12}$.
On the other hand, for all $t \in [0.15,1]$, it holds that
\begin{align*}
    g_1(1,t)+g_2(1,t) &=  \left(\frac{1}{960000}\right)(-8640 t^3 - 28056 t^2 + 32376 t + 64850)
    \\ & \leq \left(\frac{1}{960000}\right)(- 28056 t^2 + 32376 t + 64850)
    \\ & =\left(\frac{1}{960000}\right)(- 28056 t^2 + 32376 t + 64850 )
    \\ & =\frac{10841057}{140280000} - \frac{1169 (-1349/2338 + t)^2}{40000} < 1/12.
    \end{align*}
The Intermediate Value Theorem yields that for every $t\in[0.15,1]$
there exists $s\in[0,1]$ such that $g_1(s,t)+g_2(s,t)=\frac{1}{12}$.
The uniqueness of $s$ follows from Lemma~\ref{lem:dec}.
\end{proof}

We now prove that there exist $s,t\in[0,1]$ such that $g_1(s,t)=g_2(s,t)=1/24$. 

\begin{lemma}
\label{lem:st2}
There exist $s,t\in[0,1]$ such that $g_1(s,t)=g_2(s,t)=1/24$. 
\end{lemma}

\begin{proof}
For $t\in [0.15,1]$,
let $s(t)$ be the unique element of $[0,1]$ such that $g_1(s(t),t)+g_2(s(t),t)=1/12$,
which exists by Lemma~\ref{lem:sum1/12second}.
Let $h:[0.15,1]\to[0,1]$ be defined as $h(t)=g_1(s(t),t)$.
The uniqueness of $s(t)$ and continuity of the function $g_1$ and $g_2$ implies that
$s(t)$ is a continuous function from $[0.15,1]$ to $[0,1]$ and
that $h(t)$ is a continuous function as well.
We will prove that there exists $t_0\in[0.15,1]$ such that $h(t_0)=1/24$. 

We show that $h(0.15)<1/24$ and $h(1)>1/24$.
Substituting $t=0.15$ into the expressions for $g_1$ and $g_2$ yields 
\[g_1(s,0.15)= \frac{358947 s^4 - 9153522 s^3 + 10900000 s^2- 18802600 s+25186500}{480000000},\mbox{ and}\]
\[g_2(s,0.15)= \frac{358947 s^4  - 1147662 s^3 - 13046920 s^2 + 18554800 s+ 21314500}{480000000}.\]
Since $g_1(0.7,0.15)+g_2(0.7,0.15)=\frac{209573047187}{2400000000000}>\frac{1}{12}$,
Lemma~\ref{lem:dec} yields that $s(0.15)>0.7$. A direct computation yields
\begin{align}
\frac{\partial}{\partial s}g_1(s,0.15)&=\frac{1435788 s^3 - 27460566 s^2 + 21800000 s-18802600}{480000000}\nonumber\\
&\leq \frac{-26024778s^2 + 21800000 s-18802600}{480000000},\label{eq:st2g1}
\end{align} 
where the inequality uses that $s\in[0,1]$.
The expression \eqref{eq:st2g1} is a quadratic function which is negative at $s=0$;
since the associated quadratic equation has negative discriminant,
the expression \eqref{eq:st2g1} is negative for all $s$.
Therefore, for $s\in[0.7,1]$, the function $g_1(s,0.15)$ is maximized at $s=0.7$.
So, in particular, $h(0.15)=g_1(s(0.15),0.15)\leq g_1(0.7,0.15)=\frac{47707350429}{1600000000000}<\frac{1}{12}$. 

We next substitute $t=1$ into the expressions for $g_1(s,t)$ and $g_2(s,t)$ to get
\[g_1(s,1)= \left(\frac{1}{960000}\right)\left(198 s^4 - 9900 s^3 + 6602 s^2 + 13252 s +50373\right),\]
\[g_2(s,1)= \left(\frac{1}{960000}\right)\left(198 s^4 - 8460 s^3 + 24490 s^2 - 58852 s+42629\right).\]
Since $g_1(0.4,1)+g_2(0.4,1)=\frac{24553693}{300000000}<\frac{1}{12}$,
Lemma~\ref{lem:dec} yields that $s(1)<0.4$. We directly compute that
\begin{align*}
\frac{\partial}{\partial s}g_1(s,1)&=\left(\frac{1}{960000}\right)\left(792 s^3 - 29700 s^2 + 13204 s + 13252\right)\\
&\leq \left(\frac{1}{960000}\right)\left(-28908 s^2 + 13204 s + 13252\right)
\end{align*}
where the inequality uses that $s\in[0,1]$.
The final expression above is a quadratic function which is positive for all $s\in[0,0.4]$.
Therefore,
the minimum of $g_1(s,1)$ for $s\in[0,0.4]$ is attained at $s=0$ and so $h(1)=g_1(s(1),1)>g_1(0,1)=\frac{16791}{320000}>1/24$. 

We have shown that $h(0.15)<1/24$ and $h(1)>1/24$.
Since $h(t)$ is a continuous function, there exists $t_0\in[0.15,1]$ such that $h(t_0)=1/24$ by the Intermediate Value Theorem.
It follows that $g_1(s(t_0),t_0)=g_2(s(t_0),t_0)=1/24$ and the lemma is proven. 
\end{proof}

\begin{proof}[Proof of Theorem~\ref{thm:special2}]
Let $s$ and $t$ be as in Lemma~\ref{lem:st2} and take $\mu:=\mu_{st,s(1-t),1-s}$. Note that
\[d\left(4231,\mu\right)=d\left(1324,\mu\right)=g_1(s,t)=1/24\mbox{ and } d\left(2413,\mu\right)=d\left(3142,\mu\right)=g_2(s,t)=1/24.\]
Since $\mu$ is not the uniform measure on $[0,1]^2$, the quadruple $1324,2413,4132,4231$ is not quasirandom-forcing. 
\end{proof}

\section{Concluding remarks}
\label{sec:concl}

Our research is motivated by the search for a minimum quasirandom-forcing set of $4$-point permutations.
We have shown that any such set must have at least five elements.
As discussed in the introduction,
the results of~\cite{BerD14,ChaKNPSV20} provide examples of such sets consisting of eight $4$-point permutations.
We conjecture that the upper bound is tight. 

\begin{conjecture}
Every quasirandom-forcing set of $4$-point permutations has cardinality at least eight. 
\end{conjecture}

It is also natural to consider quasirandom-forcing sets of permutations of arbitrary (possibly different) sizes.
By the results of~\cite{Kur22,CruDN},
the smallest quasirandom-forcing set of permutations has cardinality between four and six.
We again conjecture that the upper bound is optimal.
This can be seen as a strengthening of~\cite[Conjecture~5.3]{CruDN},
which asserts that there is no linear combination of six permutations that
would force quasirandomness.

\begin{conjecture}
Every quasirandom-forcing set of permutations has cardinality at least six. 
\end{conjecture}

\section*{Acknowledgements}

The first author would like to thank Dan\'o Kor\'andi, Martin Kure\v cka, Samuel Mohr and Michal Stan\'\i{}k
for in-depth discussions concerning quasirandom-forcing sets of permutations. Additionally, the third author would like to thank Peter Dukes for stimulating conversations on topics related to those covered in this paper. 

\bibliographystyle{bibstyle}
\bibliography{qperm4}

\newpage

\section*{Appendix}

\subsection*{Hessian matrices used in the proof of Theorem~\ref{thm:one}}

\noindent {\bf The quadruple $1234$, $2143$, $3421$ and $4312$}\\

\noindent The gradients:

\[
\scalemath{0.21}{
\begin{array}{cccrrrrrrrrrrrrrrrrrrrrrrrrrrrrrrrrrrrrc}
  \nabla h^7_{1234}(\vec 0) & = & [ & 36358/3 & 21826/3 & 10150/3 & 1330/3 & -4634/3 & -7742/3 & 21826/3 & 16582/3 & 11314/3 & 6022/3 & 706/3 & -4634/3 & 10150/3 & 11314/3 & 11014/3 & 9250/3 & 6022/3 & 1330/3 & 1330/3 & 6022/3 & 9250/3 & 11014/3 & 11314/3 & 10150/3 & -4634/3 & 706/3 & 6022/3 & 11314/3 & 16582/3 & 21826/3 & -7742/3 & -4634/3 & 1330/3 & 10150/3 & 21826/3 & 36358/3 & ] \\
  \nabla h^7_{2143}(\vec 0) & = & [ & -31458/3 & -13566/3 & -210/3 & 8610/3 & 12894/3 & 12642/3 & -13566/3 & -2658/3 & 5442/3 & 10734/3 & 13218/3 & 12894/3 & -210/3 & 5442/3 & 9150/3 & 10914/3 & 10734/3 & 8610/3 & 8610/3 & 10734/3 & 10914/3 & 9150/3 & 5442/3 & -210/3 & 12894/3 & 13218/3 & 10734/3 & 5442/3 & -2658/3 & -13566/3 & 12642/3 & 12894/3 & 8610/3 & -210/3 & -13566/3 & -31458/3 & ] \\
  \nabla h^7_{3421}(\vec 0) & = & [ & -2450/3 & -7742/3 & -7154/3 & -686/3 & 11662/3 & 29890/3 & -518/3 & -6962/3 & -8966/3 & -6530/3 & 346/3 & 11662/3 & -2786/3 & -7790/3 & -10082/3 & -9662/3 & -6530/3 & -686/3 & -9254/3 & -10226/3 & -10502/3 & -10082/3 & -8966/3 & -7154/3 & -19922/3 & -14270/3 & -10226/3 & -7790/3 & -6962/3 & -7742/3 & -34790/3 & -19922/3 & -9254/3 & -2786/3 & -518/3 & -2450/3 & ] \\
  \nabla h^7_{4312}(\vec 0) & = & [ & -2450/3 & -518/3 & -2786/3 & -9254/3 & -19922/3 & -34790/3 & -7742/3 & -6962/3 & -7790/3 & -10226/3 & -14270/3 & -19922/3 & -7154/3 & -8966/3 & -10082/3 & -10502/3 & -10226/3 & -9254/3 & -686/3 & -6530/3 & -9662/3 & -10082/3 & -7790/3 & -2786/3 & 11662/3 & 346/3 & -6530/3 & -8966/3 & -6962/3 & -518/3 & 29890/3 & 11662/3 & -686/3 & -7154/3 & -7742/3 & -2450/3 & ] \\
  \end{array}%
}
\]

\noindent Numerical values of eigenvalues:
  $-1283.83$,
  $-1283.83$,
   $-565.96$,
   $-565.80$,
   $-553.88$,
   $-553.73$,
   $-165.77$,
   $-165.12$,
   $-162.30$,
   $-161.52$,
    $-91.65$,
    $-91.57$,
    $ 86.51$,
    $102.54$,
    $108.67$,
    $112.46$,
    $229.38$,
    $266.98$,
    $395.05$,
    $484.74$,
    $663.17$,
    $689.72$,
    $707.49$,
    $790.48$,
    $820.55$,
    $824.96$,
    $831.71$,
    $835.18$,
   $1145.41$,
   $1150.21$,
   $1389.39$,
   $1391.84$,
   $1420.41$,
   $1421.93$,
   $1944.08$,
   $1944.09$.

\newpage

\noindent {\bf The quadruple $1234$, $2341$, $3412$ and $4123$}\\

\noindent The gradients:

\[
\scalemath{0.45}{
  \begin{array}{cccrrrrrrrrrrrrrrrrc}
  \nabla h^5_{1234}(\vec 0) & = & [ & 6790/3 & 3250/3 & 550/3 & -1310/3 & 3250/3 & 2950/3 & 2050/3 & 550/3 & 550/3 & 2050/3 & 2950/3 & 3250/3 & -1310/3 & 550/3 & 3250/3 & 6790/3 & ] \\
  \nabla h^5_{2341}(\vec 0) & = & [ & -1682/3 & 2170/3 & 3550/3 & 2458/3 & -2750/3 & -50/3 & 2050/3 & 3550/3 & -4130/3 & -2150/3 & -50/3 & 2170/3 & -5822/3 & -4130/3 & -2750/3 & -1682/3 & ] \\
  \nabla h^5_{3412}(\vec 0) & = & [ & -3426/3 & -2670/3 & 30/3 & 4674/3 & -2670/3 & -2850/3 & -1950/3 & 30/3 & 30/3 & -1950/3 & -2850/3 & -2670/3 & 4674/3 & 30/3 & -2670/3 & -3426/3 & ] \\
  \nabla h^5_{4123}(\vec 0) & = & [ & -1682/3 & -2750/3 & -4130/3 & -5822/3 & 2170/3 & -50/3 & -2150/3 & -4130/3 & 3550/3 & 2050/3 & -50/3 & -2750/3 & 2458/3 & 3550/3 & 2170/3 & -1682/3 & ] \\
  \end{array}%
}
\]

\noindent Numerical values of eigenvalues:
  $-1720.86$,
  $-1717.92$,
   $-641.16$,
   $-636.00$,
   $-632.97$,
   $-613.47$,
   $-270.60$,
   $-193.37$,
    $447.65$,
    $603.77$,
   $1261.14$,
   $1290.12$,
   $1292.97$,
   $1366.06$,
   $3209.71$,
   $3226.92$.

\newpage

\noindent {\bf The quadruple $1234$, $2413$, $3142$ and $4321$}\\

\noindent The gradients:

\[
\scalemath{0.45}{
  \begin{array}{cccrrrrrrrrrrrrrrrrc}
  \nabla h^5_{1234}(\vec 0) & = & [ & 6790/3 & 3250/3 & 550/3 & -1310/3 & 3250/3 & 2950/3 & 2050/3 & 550/3 & 550/3 & 2050/3 & 2950/3 & 3250/3 & -1310/3 & 550/3 & 3250/3 & 6790/3 & ] \\
  \nabla h^5_{2413}(\vec 0) & = & [ & -4050/3 & 450/3 & 3150/3 & 4050/3 & -3150/3 & -450/3 & 450/3 & -450/3 & -450/3 & 450/3 & -450/3 & -3150/3 & 4050/3 & 3150/3 & 450/3 & -4050/3 & ] \\
  \nabla h^5_{3142}(\vec 0) & = & [ & -4050/3 & -3150/3 & -450/3 & 4050/3 & 450/3 & -450/3 & 450/3 & 3150/3 & 3150/3 & 450/3 & -450/3 & 450/3 & 4050/3 & -450/3 & -3150/3 & -4050/3 & ] \\
  \nabla h^5_{4321}(\vec 0) & = & [ & 1310/3 & -550/3 & -3250/3 & -6790/3 & -550/3 & -2050/3 & -2950/3 & -3250/3 & -3250/3 & -2950/3 & -2050/3 & -550/3 & -6790/3 & -3250/3 & -550/3 & 1310/3 & ] \\
  \end{array}%
}
\]

\noindent Numerical values of eigenvalues:
  $-332.50$,
  $-294.81$,
  $-294.81$,
  $-181.44$,
  $-181.44$,
   $-32.50$,
   $-12.50$,
   $-10.33$,
    $32.02$,
   $161.44$,
   $161.44$,
   $262.32$,
   $262.32$,
   $287.50$,
   $337.48$,
   $347.83$.

\newpage

\noindent {\bf The quadruple $1243$, $2134$, $3421$ and $4312$}\\

\noindent The gradients:

\[
\scalemath{0.45}{
  \begin{array}{cccrrrrrrrrrrrrrrrrc}
  \nabla h^5_{1243}(\vec 0) & = & [ & 6278/3 & 2930/3 & 1190/3 & 1058/3 & 2930/3 & 2750/3 & 2450/3 & 2030/3 & 1190/3 & 2450/3 & 2150/3 & 290/3 & 1058/3 & 2030/3 & 290/3 & -4162/3 & ] \\
  \nabla h^5_{2134}(\vec 0) & = & [ & -4162/3 & 290/3 & 2030/3 & 1058/3 & 290/3 & 2150/3 & 2450/3 & 1190/3 & 2030/3 & 2450/3 & 2750/3 & 2930/3 & 1058/3 & 1190/3 & 2930/3 & 6278/3 & ] \\
  \nabla h^5_{3421}(\vec 0) & = & [ & -1058/3 & -2030/3 & -290/3 & 4162/3 & -1190/3 & -2450/3 & -2150/3 & -290/3 & -2930/3 & -2750/3 & -2450/3 & -2030/3 & -6278/3 & -2930/3 & -1190/3 & -1058/3 & ] \\
  \nabla h^5_{4312}(\vec 0) & = & [ & -1058/3 & -1190/3 & -2930/3 & -6278/3 & -2030/3 & -2450/3 & -2750/3 & -2930/3 & -290/3 & -2150/3 & -2450/3 & -1190/3 & 4162/3 & -290/3 & -2030/3 & -1058/3 & ] \\
  \end{array}%
}
\]

\noindent Numerical values of eigenvalues:
  $-533.16$,
  $-407.86$,
  $-407.86$,
  $-312.01$,
  $-275.00$,
  $-275.00$,
  $-210.37$,
  $-210.37$,
  $-141.84$,
   $246.10$,
   $246.10$,
   $280.07$,
   $365.00$,
   $425.94$,
   $477.13$,
   $477.13$.

\newpage

\noindent {\bf The quadruple $1243$, $2314$, $3421$ and $4132$}\\

\noindent The gradients:

\[
\scalemath{0.45}{
  \begin{array}{cccrrrrrrrrrrrrrrrrc}
  \nabla h^5_{1243}(\vec 0) & = & [ & 6278/3 & 2930/3 & 1190/3 & 1058/3 & 2930/3 & 2750/3 & 2450/3 & 2030/3 & 1190/3 & 2450/3 & 2150/3 & 290/3 & 1058/3 & 2030/3 & 290/3 & -4162/3 & ] \\
  \nabla h^5_{2314}(\vec 0) & = & [ & -3346/3 & 890/3 & 2270/3 & 794/3 & -3790/3 & -850/3 & 1250/3 & 2510/3 & -2050/3 & -550/3 & 1550/3 & 4250/3 & 1874/3 & 1790/3 & 3170/3 & 6014/3 & ] \\
  \nabla h^5_{3421}(\vec 0) & = & [ & -1058/3 & -2030/3 & -290/3 & 4162/3 & -1190/3 & -2450/3 & -2150/3 & -290/3 & -2930/3 & -2750/3 & -2450/3 & -2030/3 & -6278/3 & -2930/3 & -1190/3 & -1058/3 & ] \\
  \nabla h^5_{4132}(\vec 0) & = & [ & -1874/3 & -1790/3 & -3170/3 & -6014/3 & 2050/3 & 550/3 & -1550/3 & -4250/3 & 3790/3 & 850/3 & -1250/3 & -2510/3 & 3346/3 & -890/3 & -2270/3 & -794/3 & ] \\
  \end{array}%
}
\]

\noindent Numerical values of eigenvalues:
  $-979.23$,
  $-944.53$,
  $-722.35$,
  $-658.15$,
  $-451.50$,
  $-361.25$,
  $-358.93$,
  $-169.78$,
    $90.14$,
   $116.26$,
   $186.00$,
   $282.33$,
   $285.47$,
   $394.18$,
   $713.30$,
   $786.05$.

\newpage

\noindent {\bf The quadruple $1324$, $2143$, $3412$ and $4231$}\\

\noindent The gradients:

\[
\scalemath{0.45}{
  \begin{array}{cccrrrrrrrrrrrrrrrrc}
  \nabla h^5_{1324}(\vec 0) & = & [ & 6718/3 & 3610/3 & 910/3 & -1382/3 & 3610/3 & 1150/3 & 250/3 & 910/3 & 910/3 & 250/3 & 1150/3 & 3610/3 & -1382/3 & 910/3 & 3610/3 & 6718/3 & ] \\
  \nabla h^5_{2143}(\vec 0) & = & [ & -4674/3 & -30/3 & 2670/3 & 3426/3 & -30/3 & 1950/3 & 2850/3 & 2670/3 & 2670/3 & 2850/3 & 1950/3 & -30/3 & 3426/3 & 2670/3 & -30/3 & -4674/3 & ] \\
  \nabla h^5_{3412}(\vec 0) & = & [ & -3426/3 & -2670/3 & 30/3 & 4674/3 & -2670/3 & -2850/3 & -1950/3 & 30/3 & 30/3 & -1950/3 & -2850/3 & -2670/3 & 4674/3 & 30/3 & -2670/3 & -3426/3 & ] \\
  \nabla h^5_{4231}(\vec 0) & = & [ & 1382/3 & -910/3 & -3610/3 & -6718/3 & -910/3 & -250/3 & -1150/3 & -3610/3 & -3610/3 & -1150/3 & -250/3 & -910/3 & -6718/3 & -3610/3 & -910/3 & 1382/3 & ] \\
  \end{array}
}
\]

\noindent Numerical values of eigenvalues:
  $-332.50$,
  $-294.81$,
  $-294.81$,
  $-181.44$,
  $-181.44$,
   $-32.50$,
   $-12.50$,
   $-10.33$,
    $32.02$,
   $161.44$,
   $161.44$,
   $262.32$,
   $262.32$,
   $287.50$,
   $337.48$,
   $347.83$.

\newpage

\noindent {\bf The quadruple $1324$, $2413$, $3241$ and $4132$}\\

\noindent The gradients:

\[
\scalemath{0.45}{
  \begin{array}{cccrrrrrrrrrrrrrrrrc}
  \nabla h^5_{1324}(\vec 0) & = & [ & 6718/3 & 3610/3 & 910/3 & -1382/3 & 3610/3 & 1150/3 & 250/3 & 910/3 & 910/3 & 250/3 & 1150/3 & 3610/3 & -1382/3 & 910/3 & 3610/3 & 6718/3 & ] \\
  \nabla h^5_{2413}(\vec 0) & = & [ & -4050/3 & 450/3 & 3150/3 & 4050/3 & -3150/3 & -450/3 & 450/3 & -450/3 & -450/3 & 450/3 & -450/3 & -3150/3 & 4050/3 & 3150/3 & 450/3 & -4050/3 & ] \\
  \nabla h^5_{3241}(\vec 0) & = & [ & -794/3 & -2270/3 & -890/3 & 3346/3 & -2510/3 & -1250/3 & 850/3 & 3790/3 & -4250/3 & -1550/3 & 550/3 & 2050/3 & -6014/3 & -3170/3 & -1790/3 & -1874/3 & ] \\
  \nabla h^5_{4132}(\vec 0) & = & [ & -1874/3 & -1790/3 & -3170/3 & -6014/3 & 2050/3 & 550/3 & -1550/3 & -4250/3 & 3790/3 & 850/3 & -1250/3 & -2510/3 & 3346/3 & -890/3 & -2270/3 & -794/3 & ] \\
  \end{array}
}
\]

\noindent Numerical values of eigenvalues:
  $-1241.61$,
  $-1238.71$,
   $-843.70$,
   $-795.33$,
   $-528.51$,
   $-523.76$,
   $-404.30$,
   $-306.97$,
    $012.35$,
    $017.12$,
    $053.14$,
    $054.09$,
    $167.07$,
    $170.40$,
    $464.15$,
    $464.58$.

\newpage

\noindent {\bf The quadruple $1342$, $2431$, $3124$ and $4213$}\\

\noindent The gradients:

\[
\scalemath{0.20}{
  \begin{array}{cccrrrrrrrrrrrrrrrrrrrrrrrrrrrrrrrrrrrrc}
  \nabla h^7_{1342}(\vec 0) & = & [ & 28910/3 & 24122/3 & 18494/3 & 12026/3 & 4718/3 & -3430/3 & 19586/3 & 14510/3 & 10754/3 & 8318/3 & 7202/3 & 7406/3 & 13454/3 & 7226/3 & 3902/3 & 3482/3 & 5966/3 & 11354/3 & 10514/3 & 2270/3 & -2062/3 & -2482/3 & 1010/3 & 8414/3 & 10766/3 & -358/3 & -7138/3 & -9574/3 & -7666/3 & -1414/3 & 14210/3 & -658/3 & -11326/3 & -17794/3 & -20062/3 & -18130/3 & ] \\
  \nabla h^7_{2431}(\vec 0) & = & [ & -14210/3 & 658/3 & 11326/3 & 17794/3 & 20062/3 & 18130/3 & -10766/3 & 358/3 & 7138/3 & 9574/3 & 7666/3 & 1414/3 & -10514/3 & -2270/3 & 2062/3 & 2482/3 & -1010/3 & -8414/3 & -13454/3 & -7226/3 & -3902/3 & -3482/3 & -5966/3 & -11354/3 & -19586/3 & -14510/3 & -10754/3 & -8318/3 & -7202/3 & -7406/3 & -28910/3 & -24122/3 & -18494/3 & -12026/3 & -4718/3 & 3430/3 & ] \\
  \nabla h^7_{3124}(\vec 0) & = & [ & -18130/3 & -20062/3 & -17794/3 & -11326/3 & -658/3 & 14210/3 & -1414/3 & -7666/3 & -9574/3 & -7138/3 & -358/3 & 10766/3 & 8414/3 & 1010/3 & -2482/3 & -2062/3 & 2270/3 & 10514/3 & 11354/3 & 5966/3 & 3482/3 & 3902/3 & 7226/3 & 13454/3 & 7406/3 & 7202/3 & 8318/3 & 10754/3 & 14510/3 & 19586/3 & -3430/3 & 4718/3 & 12026/3 & 18494/3 & 24122/3 & 28910/3 & ] \\
  \nabla h^7_{4213}(\vec 0) & = & [ & 3430/3 & -4718/3 & -12026/3 & -18494/3 & -24122/3 & -28910/3 & -7406/3 & -7202/3 & -8318/3 & -10754/3 & -14510/3 & -19586/3 & -11354/3 & -5966/3 & -3482/3 & -3902/3 & -7226/3 & -13454/3 & -8414/3 & -1010/3 & 2482/3 & 2062/3 & -2270/3 & -10514/3 & 1414/3 & 7666/3 & 9574/3 & 7138/3 & 358/3 & -10766/3 & 18130/3 & 20062/3 & 17794/3 & 11326/3 & 658/3 & -14210/3 & ] \\
  \end{array}%
}
\]

\noindent Numerical values of eigenvalues:
  $-1491.69$,
  $-1349.01$,
  $-1304.85$,
  $-1173.17$,
  $-1142.83$,
  $-1010.66$,
   $-994.72$,
   $-913.99$,
   $-913.99$,
   $-809.05$,
   $-794.85$,
   $-774.29$,
   $-719.06$,
   $-707.81$,
   $-624.90$,
   $-619.25$,
   $-615.92$,
   $-610.97$,
   $-565.13$,
   $-565.13$,
   $-491.46$,
   $-487.18$,
   $-479.56$,
   $-478.75$,
   $-423.55$,
   $-382.89$,
   $-382.89$,
   $-376.25$,
   $-301.23$,
   $-296.51$,
   $-236.74$,
    $500.35$,
    $638.21$,
    $809.22$,
   $1040.10$,
   $1194.38$.

\newpage

\noindent {\bf The quadruple $1342$, $2431$, $3214$ and $4123$}\\

\noindent The gradients:

\[
\scalemath{0.45}{
\begin{array}{cccrrrrrrrrrrrrrrrrc}
  \nabla h^5_{1342}(\vec 0) & = & [ & 6014/3 & 4250/3 & 2510/3 & 794/3 & 3170/3 & 1550/3 & 1250/3 & 2270/3 & 1790/3 & -550/3 & -850/3 & 890/3 & 1874/3 & -2050/3 & -3790/3 & -3346/3 & ] \\
  \nabla h^5_{2431}(\vec 0) & = & [ & -1874/3 & 2050/3 & 3790/3 & 3346/3 & -1790/3 & 550/3 & 850/3 & -890/3 & -3170/3 & -1550/3 & -1250/3 & -2270/3 & -6014/3 & -4250/3 & -2510/3 & -794/3 & ] \\
  \nabla h^5_{3214}(\vec 0) & = & [ & -2458/3 & -3550/3 & -2170/3 & 1682/3 & -3550/3 & -2050/3 & 50/3 & 2750/3 & -2170/3 & 50/3 & 2150/3 & 4130/3 & 1682/3 & 2750/3 & 4130/3 & 5822/3 & ] \\
  \nabla h^5_{4123}(\vec 0) & = & [ & -1682/3 & -2750/3 & -4130/3 & -5822/3 & 2170/3 & -50/3 & -2150/3 & -4130/3 & 3550/3 & 2050/3 & -50/3 & -2750/3 & 2458/3 & 3550/3 & 2170/3 & -1682/3 & ] \\
  \end{array}%
}
\]

\noindent Numerical values of eigenvalues:
  $-327.69$,
  $-316.03$,
  $-255.06$,
  $-164.14$,
   $-49.57$,
   $-49.47$,
    $51.39$,
    $65.26$,
   $144.23$,
   $228.62$,
   $251.62$,
   $282.70$,
   $287.36$,
   $316.51$,
   $404.73$,
   $409.52$.

\newpage 

\noindent {\bf The quadruple $1432$, $2341$, $3214$ and $4213$}\\

\noindent The gradients:

\[
\scalemath{0.80}{
  \begin{array}{cccrrrrrrrrrc}
  \nabla h^4_{1432}(\vec 0) & = & [ & 2024/3 & 1160/3 & 392/3 & 1160/3 & 8/3 & -1048/3 & 392/3 & -1048/3 & -1240/3 & ] \\
  \nabla h^4_{2341}(\vec 0) & = & [ & -392/3 & 1048/3 & 1240/3 & -1160/3 & -8/3 & 1048/3 & -2024/3 & -1160/3 & -392/3 & ] \\
  \nabla h^4_{3214}(\vec 0) & = & [ & -1240/3 & -1048/3 & 392/3 & -1048/3 & 8/3 & 1160/3 & 392/3 & 1160/3 & 2024/3 & ] \\
  \nabla h^4_{4123}(\vec 0) & = & [ & -392/3 & -1160/3 & -2024/3 & 1048/3 & -8/3 & -1160/3 & 1240/3 & 1048/3 & -392/3 & ]
  \end{array}%
}
\]

\noindent Numerical values of eigenvalues:
 $-194.96$,
   $59.85$,
   $64.00$,
   $92.06$,
   $92.06$,
  $247.11$,
  $355.94$,
  $355.94$,
  $512.00$.

\newpage

\subsection*{Hessian matrices used in the proof of Theorem~\ref{thm:zero}}

\noindent {\bf The quadruple $1234$, $2143$, $1243$ and $2134$}\\

\noindent The gradients:

\[
\scalemath{0.45}{
  \begin{array}{cccrrrrrrrrrrrrrrrrc}
  \nabla h^5_{1234}(\vec 0) & = & [ & 6790/3 & 3250/3 & 550/3 & -1310/3 & 3250/3 & 2950/3 & 2050/3 & 550/3 & 550/3 & 2050/3 & 2950/3 & 3250/3 & -1310/3 & 550/3 & 3250/3 & 6790/3 & ] \\
  \nabla h^5_{2143}(\vec 0) & = & [ & -4674/3 & -30/3 & 2670/3 & 3426/3 & -30/3 & 1950/3 & 2850/3 & 2670/3 & 2670/3 & 2850/3 & 1950/3 & -30/3 & 3426/3 & 2670/3 & -30/3 & -4674/3 & ] \\
  \nabla h^5_{1243}(\vec 0) & = & [ & 6278/3 & 2930/3 & 1190/3 & 1058/3 & 2930/3 & 2750/3 & 2450/3 & 2030/3 & 1190/3 & 2450/3 & 2150/3 & 290/3 & 1058/3 & 2030/3 & 290/3 & -4162/3 & ] \\
  \nabla h^5_{2134}(\vec 0) & = & [ & -4162/3 & 290/3 & 2030/3 & 1058/3 & 290/3 & 2150/3 & 2450/3 & 1190/3 & 2030/3 & 2450/3 & 2750/3 & 2930/3 & 1058/3 & 1190/3 & 2930/3 & 6278/3 & ] \\
  \end{array}%
}
\]

\noindent Numerical values of eigenvalues:
  $-355.59$,
  $-187.96$,
  $-162.80$,
  $-115.49$,
  $-114.17$,
    $ 0.00$,
    $ 0.00$,
   $165.06$,
   $190.23$,
   $371.55$,
   $408.11$,
   $447.80$,
   $582.64$,
   $609.53$,
   $742.38$,
   $874.71$.

\newpage

\noindent {\bf The quadruple $1234$, $3412$, $1432$ and $3214$}\\

\noindent The gradients:

\[
\scalemath{0.45}{
  \begin{array}{cccrrrrrrrrrrrrrrrrc}
  \nabla h^5_{1234}(\vec 0) & = & [ & 6790/3 & 3250/3 & 550/3 & -1310/3 & 3250/3 & 2950/3 & 2050/3 & 550/3 & 550/3 & 2050/3 & 2950/3 & 3250/3 & -1310/3 & 550/3 & 3250/3 & 6790/3 & ] \\
  \nabla h^5_{3412}(\vec 0) & = & [ & -3426/3 & -2670/3 & 30/3 & 4674/3 & -2670/3 & -2850/3 & -1950/3 & 30/3 & 30/3 & -1950/3 & -2850/3 & -2670/3 & 4674/3 & 30/3 & -2670/3 & -3426/3 & ] \\
  \nabla h^5_{1432}(\vec 0) & = & [ & 5822/3 & 4130/3 & 2750/3 & 1682/3 & 4130/3 & 2150/3 & 50/3 & -2170/3 & 2750/3 & 50/3 & -2050/3 & -3550/3 & 1682/3 & -2170/3 & -3550/3 & -2458/3 & ] \\
  \nabla h^5_{3214}(\vec 0) & = & [ & -2458/3 & -3550/3 & -2170/3 & 1682/3 & -3550/3 & -2050/3 & 50/3 & 2750/3 & -2170/3 & 50/3 & 2150/3 & 4130/3 & 1682/3 & 2750/3 & 4130/3 & 5822/3 & ] \\
  \end{array}%
}
\]

\noindent Numerical values of eigenvalues:
  $-2454.41$,
  $-2453.95$,
   $-937.50$,
   $-937.50$,
   $-937.50$,
   $-924.25$,
   $-358.09$,
   $-224.08$,
    $358.09$,
    $560.59$,
    $937.50$,
    $937.50$,
    $937.50$,
    $970.65$,
   $2454.41$,
   $2455.05$.

\newpage

\noindent {\bf The quadruple $1234$, $4321$, $1324$ and $4231$}\\

\noindent The gradients:

\[
\scalemath{0.20}{
  \begin{array}{cccrrrrrrrrrrrrrrrrrrrrrrrrrrrrrrrrrrrrc}
  \nabla h^7_{1234}(\vec 0) & = & [ & 36358/3 & 21826/3 & 10150/3 & 1330/3 & -4634/3 & -7742/3 & 21826/3 & 16582/3 & 11314/3 & 6022/3 & 706/3 & -4634/3 & 10150/3 & 11314/3 & 11014/3 & 9250/3 & 6022/3 & 1330/3 & 1330/3 & 6022/3 & 9250/3 & 11014/3 & 11314/3 & 10150/3 & -4634/3 & 706/3 & 6022/3 & 11314/3 & 16582/3 & 21826/3 & -7742/3 & -4634/3 & 1330/3 & 10150/3 & 21826/3 & 36358/3 & ] \\
  \nabla h^7_{4321}(\vec 0) & = & [ & 7742/3 & 4634/3 & -1330/3 & -10150/3 & -21826/3 & -36358/3 & 4634/3 & -706/3 & -6022/3 & -11314/3 & -16582/3 & -21826/3 & -1330/3 & -6022/3 & -9250/3 & -11014/3 & -11314/3 & -10150/3 & -10150/3 & -11314/3 & -11014/3 & -9250/3 & -6022/3 & -1330/3 & -21826/3 & -16582/3 & -11314/3 & -6022/3 & -706/3 & 4634/3 & -36358/3 & -21826/3 & -10150/3 & -1330/3 & 4634/3 & 7742/3 & ] \\
  \nabla h^7_{1324}(\vec 0) & = & [ & 32830/3 & 24346/3 & 15694/3 & 6874/3 & -2114/3 & -11270/3 & 24346/3 & 14782/3 & 7354/3 & 2062/3 & -1094/3 & -2114/3 & 15694/3 & 7354/3 & 2302/3 & 538/3 & 2062/3 & 6874/3 & 6874/3 & 2062/3 & 538/3 & 2302/3 & 7354/3 & 15694/3 & -2114/3 & -1094/3 & 2062/3 & 7354/3 & 14782/3 & 24346/3 & -11270/3 & -2114/3 & 6874/3 & 15694/3 & 24346/3 & 32830/3 & ] \\
  \nabla h^7_{4231}(\vec 0) & = & [ & 11270/3 & 2114/3 & -6874/3 & -15694/3 & -24346/3 & -32830/3 & 2114/3 & 1094/3 & -2062/3 & -7354/3 & -14782/3 & -24346/3 & -6874/3 & -2062/3 & -538/3 & -2302/3 & -7354/3 & -15694/3 & -15694/3 & -7354/3 & -2302/3 & -538/3 & -2062/3 & -6874/3 & -24346/3 & -14782/3 & -7354/3 & -2062/3 & 1094/3 & 2114/3 & -32830/3 & -24346/3 & -15694/3 & -6874/3 & 2114/3 & 11270/3 & ] \\
  \end{array}%
}
\]

\noindent Numerical values of eigenvalues:
   $-680.77$,
   $-671.30$,
   $-580.69$,
   $-580.69$,
   $-573.34$,
   $-573.34$,
   $-471.70$,
   $-465.71$,
    $362.32$,
    $400.84$,
    $400.84$,
    $450.39$,
    $541.17$,
    $629.28$,
    $630.65$,
    $630.65$,
    $713.82$,
    $728.15$,
    $766.61$,
    $766.61$,
    $798.54$,
    $798.54$,
    $886.38$,
    $938.16$,
   $1039.08$,
   $1199.48$,
   $1199.48$,
   $1383.62$,
   $2336.01$,
   $2463.89$,
   $2480.92$,
   $2480.92$,
   $2615.50$,
   $2615.50$,
   $2648.47$,
   $2795.73$.

\newpage

\noindent {\bf The quadruple $1243$, $3421$, $1423$ and $3241$}\\

\noindent The gradients:

\[
\scalemath{0.45}{
  \begin{array}{cccrrrrrrrrrrrrrrrrc}
  \nabla h^5_{1243}(\vec 0) & = & [ & 6278/3 & 2930/3 & 1190/3 & 1058/3 & 2930/3 & 2750/3 & 2450/3 & 2030/3 & 1190/3 & 2450/3 & 2150/3 & 290/3 & 1058/3 & 2030/3 & 290/3 & -4162/3 & ] \\
  \nabla h^5_{3421}(\vec 0) & = & [ & -1058/3 & -2030/3 & -290/3 & 4162/3 & -1190/3 & -2450/3 & -2150/3 & -290/3 & -2930/3 & -2750/3 & -2450/3 & -2030/3 & -6278/3 & -2930/3 & -1190/3 & -1058/3 & ] \\
  \nabla h^5_{1423}(\vec 0) & = & [ & 6014/3 & 3170/3 & 1790/3 & 1874/3 & 4250/3 & 1550/3 & -550/3 & -2050/3 & 2510/3 & 1250/3 & -850/3 & -3790/3 & 794/3 & 2270/3 & 890/3 & -3346/3 & ] \\
  \nabla h^5_{3241}(\vec 0) & = & [ & -794/3 & -2270/3 & -890/3 & 3346/3 & -2510/3 & -1250/3 & 850/3 & 3790/3 & -4250/3 & -1550/3 & 550/3 & 2050/3 & -6014/3 & -3170/3 & -1790/3 & -1874/3 & ] \\
  \end{array}%
}
\]

\noindent Numerical values of eigenvalues:
  $-602.21$,
  $-601.85$,
  $-600.39$,
  $-599.60$,
  $-140.40$,
  $-132.64$,
  $-102.55$,
   $-87.22$,
   $377.75$,
   $398.12$,
   $492.45$,
   $537.69$,
   $619.14$,
   $622.20$,
   $652.21$,
   $703.28$.

\newpage

\noindent {\bf The quadruple $1324$, $2413$, $1423$ and $2314$}\\

\noindent The gradients:

\[
\scalemath{0.45}{
  \begin{array}{cccrrrrrrrrrrrrrrrrc}
  \nabla h^5_{1324}(\vec 0) & = & [ & 6718/3 & 3610/3 & 910/3 & -1382/3 & 3610/3 & 1150/3 & 250/3 & 910/3 & 910/3 & 250/3 & 1150/3 & 3610/3 & -1382/3 & 910/3 & 3610/3 & 6718/3 & ] \\
  \nabla h^5_{2413}(\vec 0) & = & [ & -4050/3 & 450/3 & 3150/3 & 4050/3 & -3150/3 & -450/3 & 450/3 & -450/3 & -450/3 & 450/3 & -450/3 & -3150/3 & 4050/3 & 3150/3 & 450/3 & -4050/3 & ] \\
  \nabla h^5_{1423}(\vec 0) & = & [ & 6014/3 & 3170/3 & 1790/3 & 1874/3 & 4250/3 & 1550/3 & -550/3 & -2050/3 & 2510/3 & 1250/3 & -850/3 & -3790/3 & 794/3 & 2270/3 & 890/3 & -3346/3 & ] \\
  \nabla h^5_{2314}(\vec 0) & = & [ & -3346/3 & 890/3 & 2270/3 & 794/3 & -3790/3 & -850/3 & 1250/3 & 2510/3 & -2050/3 & -550/3 & 1550/3 & 4250/3 & 1874/3 & 1790/3 & 3170/3 & 6014/3 & ] \\
  \end{array}%
}
\]

\noindent Numerical values of eigenvalues:
  $-870.83$,
  $-840.02$,
  $-642.58$,
  $-545.00$,
  $-449.62$,
  $-411.16$,
  $-344.50$,
  $-175.38$,
    $65.65$,
    $85.04$,
   $163.32$,
   $323.95$,
   $417.39$,
   $423.44$,
   $821.17$,
   $827.12$.

\newpage

\noindent {\bf The quadruple $1342$, $2431$, $1432$ and $2341$}\\

\noindent The gradients:

\[
\scalemath{0.45}{
  \begin{array}{cccrrrrrrrrrrrrrrrrc}
  \nabla h^5_{1342}(\vec 0) & = & [ & 6014/3 & 4250/3 & 2510/3 & 794/3 & 3170/3 & 1550/3 & 1250/3 & 2270/3 & 1790/3 & -550/3 & -850/3 & 890/3 & 1874/3 & -2050/3 & -3790/3 & -3346/3 & ] \\
  \nabla h^5_{2431}(\vec 0) & = & [ & -1874/3 & 2050/3 & 3790/3 & 3346/3 & -1790/3 & 550/3 & 850/3 & -890/3 & -3170/3 & -1550/3 & -1250/3 & -2270/3 & -6014/3 & -4250/3 & -2510/3 & -794/3 & ] \\
  \nabla h^5_{1432}(\vec 0) & = & [ & 5822/3 & 4130/3 & 2750/3 & 1682/3 & 4130/3 & 2150/3 & 50/3 & -2170/3 & 2750/3 & 50/3 & -2050/3 & -3550/3 & 1682/3 & -2170/3 & -3550/3 & -2458/3 & ] \\
  \nabla h^5_{2341}(\vec 0) & = & [ & -1682/3 & 2170/3 & 3550/3 & 2458/3 & -2750/3 & -50/3 & 2050/3 & 3550/3 & -4130/3 & -2150/3 & -50/3 & 2170/3 & -5822/3 & -4130/3 & -2750/3 & -1682/3 & ] \\
  \end{array}%
}
\]

\noindent Numerical values of eigenvalues:
  $-847.86$,
  $-697.89$,
  $-690.03$,
  $-585.98$,
  $-494.21$,
  $-423.28$,
  $-422.32$,
  $-373.33$,
  $-217.50$,
  $-172.21$,
  $-164.24$,
  $-137.42$,
   $112.53$,
   $115.10$,
   $181.60$,
   $209.03$.

\newpage

\noindent {\bf The quadruple $2143$, $3412$, $2413$ and $3142$}\\

\noindent The gradients:

\[
\scalemath{0.25}{
  \begin{array}{cccrrrrrrrrrrrrrrrrrrrrrrrrrrrrrrrrrrrrc}
  \nabla h^7_{2143}(\vec 0) & = & [ & -10486 & -4522 & -70 & 2870 & 4298 & 4214 & -4522 & -886 & 1814 & 3578 & 4406 & 4298 & -70 & 1814 & 3050 & 3638 & 3578 & 2870 & 2870 & 3578 & 3638 & 3050 & 1814 & -70 & 4298 & 4406 & 3578 & 1814 & -886 & -4522 & 4214 & 4298 & 2870 & -70 & -4522 & -10486 & ] \\
  \nabla h^7_{3412}(\vec 0) & = & [ & -4214 & -4298 & -2870 & 70 & 4522 & 10486 & -4298 & -4406 & -3578 & -1814 & 886 & 4522 & -2870 & -3578 & -3638 & -3050 & -1814 & 70 & 70 & -1814 & -3050 & -3638 & -3578 & -2870 & 4522 & 886 & -1814 & -3578 & -4406 & -4298 & 10486 & 4522 & 70 & -2870 & -4298 & -4214 & ] \\
  \nabla h^7_{2413}(\vec 0) & = & [ & -7350 & -2058 & 2058 & 4998 & 6762 & 7350 & -6762 & -2646 & 294 & 2058 & 2646 & 2058 & -4998 & -2058 & -294 & 294 & -294 & -2058 & -2058 & -294 & 294 & -294 & -2058 & -4998 & 2058 & 2646 & 2058 & 294 & -2646 & -6762 & 7350 & 6762 & 4998 & 2058 & -2058 & -7350 & ] \\
  \nabla h^7_{3142}(\vec 0) & = & [ & -7350 & -6762 & -4998 & -2058 & 2058 & 7350 & -2058 & -2646 & -2058 & -294 & 2646 & 6762 & 2058 & 294 & -294 & 294 & 2058 & 4998 & 4998 & 2058 & 294 & -294 & 294 & 2058 & 6762 & 2646 & -294 & -2058 & -2646 & -2058 & 7350 & 2058 & -2058 & -4998 & -6762 & -7350 & ] \\
  \end{array}%
}
\]

\noindent Numerical values of eigenvalues:
   $-680.34$,
   $-671.74$,
   $-579.95$,
   $-579.95$,
   $-574.08$,
   $-574.08$,
   $-470.65$,
   $-466.74$,
    $361.89$,
    $400.31$,
    $400.31$,
    $449.65$,
    $555.45$,
    $602.06$,
    $651.92$,
    $651.92$,
    $708.70$,
    $708.70$,
    $729.27$,
    $765.80$,
    $803.38$,
    $839.22$,
    $839.22$,
    $990.27$,
   $1036.90$,
   $1195.73$,
   $1195.73$,
   $1377.13$,
   $2336.41$,
   $2463.59$,
   $2481.75$,
   $2481.75$,
   $2614.90$,
   $2614.90$,
   $2650.08$,
   $2794.58$.

\end{document}